\documentclass[a4paper,12pt]{article}

\usepackage[usenames,dvipsnames]{pstricks}

\usepackage{amscd,amssymb,amsmath,graphicx,verbatim,hyperref, stmaryrd}
\usepackage[OT1,T1]{fontenc}

\usepackage[all, cmtip]{xy}
\usepackage{pdfpages}

\usepackage{amsmath,amsthm,amsfonts,mathabx,tikz,tikz-cd,stmaryrd,url,mathtools, mathrsfs}
\usepackage{hyperref}
\hypersetup{
	colorlinks,
	citecolor=black,
	filecolor=black,
	linkcolor=black,
	urlcolor=black
}

\usetikzlibrary{matrix,arrows,decorations.pathmorphing}

\newcounter{Definitioncount}
\newtheorem{theorem}{Theorem}% with counter

\newtheorem{proposition}[theorem]{Proposition}
\newtheorem{corollary}[theorem]{Corollary}
\theoremstyle{definition}
\newtheorem{remark}[theorem]{Remark}
\newtheorem*{Example}{Example}% no counter

\newtheoremstyle{fact}{\bigskipamount}{\medskipamount}{\upshape}{}{\itshape}{. }{ }{Fact}
\theoremstyle{fact}
\newtheoremstyle{genquest}{\bigskipamount}{\medskipamount}{\upshape}{}{\itshape}{. }{ }{General Question}
\theoremstyle{genquest}
\newtheoremstyle{step}{2\bigskipamount}{\medskipamount}{\upshape}{}{\itshape}{. }{ }{\underline{Step~\thestep}}
\theoremstyle{step}

\renewcommand{\thestep}{\arabic{step}}

\newcommand{\lra}{\longrightarrow}

\newcommand{\ra}{\rightarrow}

\newcommand{\Ra}{\Rightarrow}

\newcommand{\ldual}[1]{\mathord{{\let\nolimits\relax\sideset{^\wedge}{}{#1}}}}
\newcommand{\laction}[2]{\mathord{{\let\nolimits\relax\sideset{^{#1}}{}{#2}}}}
\newcommand{\conj}[2]{\mathord{{\let\nolimits\relax\sideset{^{#1}}{}{#2}}}}

\newcommand{\ox}{\otimes}
\newcommand{\xra}{\xrightarrow}
\newcommand{\xla}{\xleftarrow}
\newcommand{\ob}{\mathrm{ob}}

\def\CA{{\mathscr A}}
\def\CB{{\mathscr B}}
\def\CC{{\mathscr C}}
\def\CD{{\mathscr D}}

\def\CF{{\mathscr F}}
\def\CG{{\mathscr G}}

\def\CJ{{\mathscr J}}

\def\CV{{\mathscr V}}

\def\CX{{\mathscr X}}

\makeatletter
\newcommand*\bigcdot{\mathpalette\bigcdot@{.5}}
\newcommand*\bigcdot@[2]{\mathbin{\vcenter{\hbox{\scalebox{#2}{$\m@th#1\bullet$}}}}}
\makeatother

\begin{document}
\author{Kevin Coulembier, Ross Street and Michel van den Bergh\footnote{The authors gratefully acknowledge the support of Australian Research Council Discovery Grants DE170100623 and DP190102432.}}

\title{Freely adjoining monoidal duals}
\date{\small{\today}}
\maketitle

\noindent {\small{\emph{2010 Mathematics Subject Classification:} 18D10; 18D35; 18A40; 16T99; 47A07; 52C05}

\begin{abstract}
\noindent Given a monoidal category $\CC$ with an object $J$, we construct a monoidal category $\CC[J^{\vee}]$ by freely adjoining a right dual $J^{\vee}$ to $J$. We show that the canonical strong monoidal functor $\Omega : \CC\to \CC[J^{\vee}]$ provides the unit for a biadjunction with the forgetful 2-functor from the 2-category of monoidal categories with a distinguished dual pair to the 2-category of monoidal categories with a distinguished object. We show that $\Omega : \CC\to \CC[J^{\vee}]$ is fully faithful and provide coend formulas
for homs of the form $\CC[J^{\vee}](U,\Omega A)$ and $\CC[J^{\vee}](\Omega A,U)$ for $A\in \CC$ and $U\in \CC[J^{\vee}]$.

If $\mathbb{N}$ denotes the free strict monoidal category on a single generating
object $1$ then $\mathbb{N}[1^{\vee}]$ is the free monoidal category $\mathrm{Dpr}$ containing a dual pair
$- \dashv +$ of objects. As we have the monoidal pseudopushout $\CC[J^{\vee}] \simeq \mathrm{Dpr} +_{\mathbb{N}} \CC$, it is of interest to have an explicit model of $\mathrm{Dpr}$: we provide both geometric
and combinatorial models. We show that the (algebraist's) simplicial category $\Delta$ is a monoidal 
full subcategory of $\mathrm{Dpr}$ and explain the relationship with the free 2-category 
$\mathrm{Adj}$ containing an adjunction. We describe a generalization of $\mathrm{Dpr}$ which includes, for example, a combinatorial model $\mathrm{Dseq}$ for the free monoidal category containing a duality sequence $X_0\dashv X_1\dashv X_2 \dashv \dots$ of objects. 
Actually, $\mathrm{Dpr}$ is a monoidal full subcategory of $\mathrm{Dseq}$.     
\end{abstract}

\section*{Introduction}

From the higher categorical point of view, dual objects in monoidal categories are a special case
of adjoint morphisms in bicategories and so the concept basically goes back to the counit-unit
definition of adjunction due to Kan \cite{Kan1958}.
The special case was made explicit for different uses by Kelly \cite{Kel1972} and Saavedra Rivano \cite{NSR1972}.
A symmetric monoidal category in which every object has a dual, Kelly called {\em compact} and
Saavedra Rivano called {\em rigid}. Much later, without the symmetry requirement and so requiring both
a left and right duals, Joyal-Street \cite{xii} used the term {\em autonomous}. 

This project began by our seeking a construction which freely adjoined dual objects to a given monoidal category $\CC$ with some chosen objects to have the duals. The existence was not in doubt.
We wanted the construction to be sufficiently explicit for us to prove various properties of it
and its relationship to $\CC$. 
While considering this, we came across the work of A. Delpeuch \cite{Delp2019} in which he uses the
string diagrams as in the
unpublished work \cite{xii} of A. Joyal and the second author to construct the free autonomous (rigid) monoidal
category $\mathrm{Auton}\CC$ on $\CC$: all objects are given duals. 
He proves that the canonical functor $\CC\to \mathrm{Auton}\CC$ is fully faithful.
This was also one of our observations. 

Since our constructions and techniques are quite different from those of \cite{Delp2019}, we feel they
should be of independent interest. Moreover, we deal with situations where fewer duals are adjoined
than in $\mathrm{Auton}\CC$ and we instigate the study of general Hom sets in such categories. We shall now describe the contents of this paper in more detail. 

After a brief review in Section~\ref{dp} of duality for objects in a monoidal category and 
associated string diagrams, we describe in Section~\ref{gmDpr}
a geometric model of the free monoidal category $\mathrm{Dpr}$ containing a dual pair $-\dashv +$. 
The objects are words in the symbols $-$ and $+$ while the morphisms are string diagrams. We
should think of these morphisms as in normal form since, when the string diagrams are composed via
vertical stacking, the result must be reduced to normal form by tugging (that is, by applying the snake identities).

The goal of Section~\ref{CdDpr} is to present a non-skeletal combinatorial model of $\mathrm{Dpr}$.
The objects are pairs $(M,S)$ where $M$ is a finite linearly ordered set and $S$ is a subset;
we think of $M$ as the set of positions for the letters $-$ and $+$ in the word and $S$ tells where there
is a $+$. Morphisms $(A,B) : (M,S) \to (N,T)$ consist of subsets $A\subseteq M$ and $B\subseteq N$,
which tell the positions in the domain that are joined to positions in the codomain by a string;
there are conditions on $(A,B)$.    

In Section~\ref{Rfa}, we show that the (algebraist's) simplicial category $\Delta$ is a monoidal 
full subcategory of $\mathrm{Dpr}$ and explain the relationship with the free 2-category 
$\mathrm{Adj}$ containing an adjunction.
   
As an interlude, in Section~\ref{Ioid}, we generalize the construction of $\mathrm{Dpr}$ to include, for example, a combinatorial model $\mathrm{Dseq}$ for the free monoidal category containing a duality sequence $X_0\dashv X_1\dashv X_2 \dashv \dots$ of objects. 

Section~\ref{ConstructionC[J*]} constructs the free monoidal category $\CC[J^{\vee}]$ adjoining a
right dual to a given object $J$ in a monoidal category $\CC$.
Section~\ref{UniversalityC[J*]} proves the universal property. 
Sections~\ref{Omegafullsection} proves that the canonical strong monoidal functor $\Omega : \CC \to \CC[J^{\vee}]$ is full by appealing to the details of the construction.
Sections~\ref{Omegafaithfulsection} proves that $\Omega : \CC \to \CC[J^{\vee}]$ is faithful by using a more transcendental argument.
Section~\ref{coends} provides formulas for some homs in $\CC[J^{\vee}]$ other than those coming from $\CC$. In Section~\ref{RemPushN}, we demonstrate that adjoining a dual corresponds to taking a pseudopushout with Dpr. In Section~\ref{Klinear} we explain how our results extend to $K$-linear categories. In Section~\ref{Auton} we discuss the link between our results and the principle of autonomisation.

\tableofcontents

\section{Dual pairs}\label{dp}

For the definition of monoidal category $\CV$ (also called tensor category) see \cite{EilKel1966, CWM, 37}. We use $\mathbb{I}$ for the tensor unit and $\ox$ for the binary tensor product. 
We will make use of the fact that every monoidal category is monoidally equivalent to a strict monoidal category (that is, one in which the associativity and unitality isomorphisms are identities).

For a linearly ordered set $M = \{m_1< m_2<\dots < m_k\}$ and $X_m\in \CV$ for all $m\in M$, we put 
$$\bigotimes_{m\in M} X_m = X_{m_1}\ox X_{m_2}\ox \dots \ox X_{m_k} \ .$$ 

Duality for objects in $\CV$ was defined in \cite{NSR1972, Kel1972, xii}.
We will make use of the string diagrams for monoidal categories as explained in \cite{xii, 37}
except that we will read from top to bottom rather than left to right or bottom to top.   

A {\em right dual} for an object $X$ of $\CV$ is an object $Y$ equipped with morphisms
$\varepsilon : X\otimes Y \to \mathbb{I}$ and $\eta : \mathbb{I} \to Y\otimes X$, called {\em counit} and {\em unit},
such that the composites 
\begin{eqnarray}\label{radj}
X\xra{X\ox \eta}X\ox Y\ox X \xra{\varepsilon \ox X}X  \ \text{ and } \  Y\xra{\eta \ox Y}Y\ox X\ox Y \xra{Y\ox \varepsilon}Y
\end{eqnarray}
are identities. As this is a special case of adjunction in a bicategory, we use the notation $X\dashv Y$ and call this a {\em dual pair}. Here are the string diagrams expressing that the composites in \eqref{radj} are identities.

\begin{eqnarray}\label{snakecon}
\begin{aligned}
\psscalebox{0.6 0.6} % Change this value to rescale the drawing.
{
\begin{pspicture}(0,-2.8526227)(17.039299,2.8526227)
\pscircle[linecolor=black, linewidth=0.04, dimen=outer](2.9392998,0.44737723){0.32}
\pscircle[linecolor=black, linewidth=0.04, dimen=outer](11.3793,0.40737724){0.32}
\pscircle[linecolor=black, linewidth=0.04, dimen=outer](1.1792998,-1.0126227){0.32}
\pscircle[linecolor=black, linewidth=0.04, dimen=outer](13.4593,-1.0326228){0.32}
\rput[bl](4.5793,-0.43262276){\LARGE${=}$}
\rput[bl](15.3193,-0.61262274){\LARGE${=}$}
\rput[bl](13.3993,-1.1726228){\large${\varepsilon}$}
\rput[bl](1.1392999,-1.1526227){\large${\varepsilon}$}
\rput[bl](2.8792999,0.28737724){\large${\eta}$}
\rput[bl](11.2993,0.26737726){\large${\eta}$}
\psline[linecolor=black, linewidth=0.04](0.01929985,2.8473773)(0.9992998,-0.7526228)
\psline[linecolor=black, linewidth=0.04](1.3792999,-0.77262276)(2.7192998,0.20737724)(2.7392998,0.18737724)
\psline[linecolor=black, linewidth=0.04](3.0993,0.18737724)(4.4593,-2.7926228)
\psline[linecolor=black, linewidth=0.04](6.2792997,2.8073773)(6.3192997,-2.8126228)
\psline[linecolor=black, linewidth=0.04](11.1793,0.14737724)(10.1993,-2.8126228)
\psline[linecolor=black, linewidth=0.04](11.599299,0.18737724)(13.2193,-0.89262277)
\psline[linecolor=black, linewidth=0.04](13.6993,-0.8126228)(15.1593,2.7473772)
\psline[linecolor=black, linewidth=0.04](17.0193,2.7673771)(17.0193,-2.8526227)
\rput[bl](5.8,2.5473773){$X$}
\rput[bl](0.31929985,2.4673772){$X$}
\rput[bl](14.5,2.3873773){$Y$}
\rput[bl](3.6193,-2.5126228){$X$}
\rput[bl](1.5592998,-0.15262276){$Y$}
\rput[bl](16.6,2.3673773){$Y$}
\rput[bl](11.7793,-0.67262274){$X$}
\rput[bl](10.5192995,-2.6126227){$Y$}
\end{pspicture}}
\end{aligned}
\end{eqnarray}

When there is no ambiguity, we denote counits by cups $\cup$ and units by caps $\cap$. So \eqref{radj} becomes the more geometrically ``obvious'' operation of pulling the ends of the strings as in \eqref{radjobv}.
These are sometimes called the {\em snake equations}.
\begin{eqnarray}\label{radjobv}
\begin{aligned}
\psscalebox{0.6 0.6} % Change this value to rescale the drawing.
{
\begin{pspicture}(0,-2.72)(15.8,2.72)
\pscustom[linecolor=black, linewidth=0.04]
{
\newpath
\moveto(2.04,1.74)
\lineto(1.98,0.98999995)
\curveto(1.95,0.61499995)(1.935,-0.050000038)(1.95,-0.34000003)
\curveto(1.965,-0.63000005)(2.02,-0.97)(2.06,-1.02)
\curveto(2.1,-1.07)(2.18,-1.1500001)(2.22,-1.1800001)
\curveto(2.26,-1.21)(2.43,-1.255)(2.56,-1.27)
\curveto(2.69,-1.2850001)(2.905,-1.27)(2.99,-1.24)
\curveto(3.075,-1.21)(3.2,-1.1500001)(3.24,-1.12)
\curveto(3.28,-1.09)(3.355,-0.87500006)(3.39,-0.69000006)
\curveto(3.425,-0.50500005)(3.495,-0.22000004)(3.53,-0.120000035)
\curveto(3.565,-0.020000039)(3.65,0.13499996)(3.7,0.18999997)
\curveto(3.75,0.24499996)(3.98,0.28999996)(4.16,0.27999997)
\curveto(4.34,0.26999995)(4.57,0.22499996)(4.62,0.18999997)
\curveto(4.67,0.15499996)(4.755,0.0149999615)(4.79,-0.09000004)
\curveto(4.825,-0.19500004)(4.89,-0.65000004)(4.92,-1.0)
\curveto(4.95,-1.35)(5.01,-1.77)(5.04,-1.84)
\curveto(5.07,-1.9100001)(5.14,-2.16)(5.26,-2.7)
}
\pscustom[linecolor=black, linewidth=0.04]
{
\newpath
\moveto(6.98,1.74)
\lineto(7.03,0.84)
\curveto(7.055,0.38999996)(7.06,-0.44500005)(7.04,-0.83000004)
\curveto(7.02,-1.215)(7.0,-1.88)(7.0,-2.72)
}
\pscustom[linecolor=black, linewidth=0.04]
{
\newpath
\moveto(9.64,-2.64)
\lineto(9.69,-2.19)
\curveto(9.715,-1.965)(9.775,-1.5300001)(9.81,-1.32)
\curveto(9.845,-1.11)(9.905,-0.68000007)(9.93,-0.46000004)
\curveto(9.955,-0.24000004)(10.025,0.009999962)(10.07,0.03999996)
\curveto(10.115,0.06999996)(10.2,0.12999997)(10.24,0.15999997)
\curveto(10.28,0.18999997)(10.575,0.11999996)(10.83,0.019999962)
\curveto(11.085,-0.080000035)(11.37,-0.37000003)(11.4,-0.56000006)
\curveto(11.43,-0.75000006)(11.49,-1.09)(11.52,-1.24)
\curveto(11.55,-1.39)(11.67,-1.585)(11.76,-1.63)
\curveto(11.85,-1.6750001)(12.035,-1.735)(12.13,-1.75)
\curveto(12.225,-1.765)(12.365,-1.745)(12.41,-1.71)
\curveto(12.455,-1.6750001)(12.54,-1.61)(12.58,-1.58)
\curveto(12.62,-1.5500001)(12.7,-1.485)(12.74,-1.45)
\curveto(12.78,-1.4150001)(12.845,-1.1500001)(12.87,-0.92)
\curveto(12.895,-0.69000006)(12.955,-0.37500003)(12.99,-0.29000005)
\curveto(13.025,-0.20500004)(13.09,0.35499996)(13.18,1.78)
}
\pscustom[linecolor=black, linewidth=0.04]
{
\newpath
\moveto(15.02,1.74)
\lineto(15.08,0.71999997)
\curveto(15.11,0.20999996)(15.13,-0.56000006)(15.12,-0.82000005)
\curveto(15.11,-1.08)(15.085,-1.575)(15.07,-1.8100001)
\curveto(15.055,-2.045)(15.055,-2.365)(15.1,-2.6200001)
}
\rput[bl](5.54,-0.74){\LARGE${=}$}
\rput[bl](13.46,-0.78000003){\LARGE${=}$}
\rput[bl](1.34,1.4599999){$X$}
\rput[bl](6.3,1.38){$X$}
\rput[bl](12.7,1.38){$Y$}
\rput[bl](2.97,-0.64000005){$Y$}
\rput[bl](10.88,-0.86){$X$}
\rput[bl](15.16,1.38){$Y$}
\rput[bl](4.56,-2.28){$X$}
\rput[bl](9.2,-2){$Y$}
\end{pspicture}
}
\end{aligned}
\end{eqnarray}

Dualities tensor. If $X\dashv Y$ and $Z\dashv W$ then $X\ox Z\dashv W\ox Y$; 
the counit and unit are the nested cups and nested caps as shown in \eqref{adjcomp}.

\begin{eqnarray}\label{adjcomp}
\begin{aligned}
\psscalebox{0.6 0.6} % Change this value to rescale the drawing.
{
\begin{pspicture}(0,-0.8380139)(8.52,0.8380139)
\psbezier[linecolor=black, linewidth=0.04](1.02,0.47792357)(1.02,-0.3220764)(2.42,-0.3220764)(2.42,0.477923583984375)
\psbezier[linecolor=black, linewidth=0.04](0.32,0.47792357)(0.32,-1.1220764)(3.1,-1.1620765)(3.1,0.437923583984375)
\psbezier[linecolor=black, linewidth=0.04](7.4002223,-0.3917126)(7.3882093,0.4081972)(5.988367,0.38717374)(6.0003805,-0.412736061473463)
\psbezier[linecolor=black, linewidth=0.04](8.100143,-0.38120085)(8.076117,1.2186188)(5.29583,1.2168676)(5.319856,-0.38295196956351446)
\rput[bl](0.0,0.5379236){$X$}
\rput[bl](2.84,0.5179236){$Y$}
\rput[bl](0.74,0.5379236){$Z$}
\rput[bl](2.1,0.5179236){$W$}
\rput[bl](7.16,-0.76207644){$X$}
\rput[bl](5.8,-0.7420764){$Y$}
\rput[bl](7.84,-0.76207644){$Z$}
\rput[bl](5.02,-0.7220764){$W$}
\end{pspicture}
}\end{aligned}
\end{eqnarray}
In particular, if $X\dashv Y$ then $X^{\ox n}\dashv Y^{\ox n}$ with counit and unit
given by $n$ nested cups and $n$ nested caps.

If $X\dashv Y$ and $A\dashv B$ in $\CV$, recall that morphisms $A\to X$ and $Y\to B$ are {\em mates} if they are exchanged under the bijection $\CV(A,X)\cong \CV(Y,B)$.

We remind the reader that duals invert.

\begin{proposition}[\cite{NSR1972, xii, DayPas2008}]
Suppose $X\dashv Y$ is a dual pair in a monoidal category $\CA$.
If $\sigma : S\Ra T : \CA \to \CV$ is a monoidal natural 
transformation between strong
monoidal natural transformations then the components $\sigma_X$ and $\sigma_Y$
of $\sigma$ are invertible. Indeed, $\sigma_{X}^{-1}$ is the mate of $\sigma_Y$
and $\sigma_{Y}^{-1}$ is the mate of $\sigma_X$ under the dualities $SX\dashv SY$ and $TX\dashv TY$.  
\end{proposition}

A morphism $(u,v) : (X, Y, \varepsilon, \eta)\to (X_1, Y_1, \varepsilon_1, \eta_1)$
of dual pairs in $\CV$ consists of morphisms $u : X\to X_1$ and $v : Y_1\to Y$
in $\CV$ such that $\varepsilon_1\circ (u\ox 1) = \varepsilon \circ (1\ox v)$; see \eqref{mordupr}. In fact, $u$ and $v$ are mates under the dualities so determine each other.
This defines a monoidal category $\mathrm{Dp}\CV$ of dual pairs in $\CV$.
\begin{eqnarray}\label{mordupr}
\begin{aligned}
\psscalebox{0.6 0.6} % Change this value to rescale the drawing.
{
\begin{pspicture}(0,-1.1304754)(5.28,1.1304754)
\psellipse[linecolor=black, linewidth=0.04, dimen=outer](0.29,-0.24047536)(0.29,0.29)
\psellipse[linecolor=black, linewidth=0.04, dimen=outer](4.73,-0.24047536)(0.29,0.29)
\psbezier[linecolor=black, linewidth=0.04](3.4,-0.49047536)(3.4,-1.2904754)(4.78,-1.2904754)(4.78,-0.4904753494262695)
\psbezier[linecolor=black, linewidth=0.04](0.28,-0.51047534)(0.28,-1.3104753)(1.66,-1.3104753)(1.66,-0.5104753494262695)
\rput[bl](2.22,-0.11047535){\LARGE${=}$}
\psline[linecolor=black, linewidth=0.04](1.66,-0.5304754)(1.74,1.1295247)
\psline[linecolor=black, linewidth=0.04](0.28,0.009524651)(0.26,1.1095246)
\psline[linecolor=black, linewidth=0.04](3.4,-0.5304754)(3.36,1.1295247)
\psline[linecolor=black, linewidth=0.04](4.76,0.02952465)(4.74,1.1295247)
\rput[bl](0.12,-0.33047536){\Large${u}$}
\rput[bl](4.60,-0.35047534){\Large${v}$}
\end{pspicture}
}
\end{aligned}
\end{eqnarray}

A monoidal category is called {\em autonomous} when every object has a left and a right dual.
The terms {\em rigid {\em and} compact} are all used in the literature to mean the existence of duals. 
The free autonomous monoidal category on any autonomous tensor scheme was described geometrically in \cite{xii}.
The free autonomous monoidal category on any monoidal category was described geometrically in \cite{Delp2019}.
Here we are interested in combinatorial descriptions of related free monoidal structures. 

\section{Geometric model of $\mathrm{Dpr}$}\label{gmDpr}
Our interest here is in the free monoidal category $\mathrm{Dpr}$ containing a duality pair of objects.
A presentation of this in terms of tensor schemes (in the sense of \cite{37})
takes the generating tensor scheme to have two objects $-$ and $+$
and two morphisms $\varepsilon : - + \to \varnothing$ and 
$\eta : \varnothing \to +-$ and subjects them to the snake equations.
The objects of $\mathrm{Dpr}$ are therefore elements of the free monoid
$\{-,+\}^*$ on the two symbols $-$ and $+$. 
As strings, each morphism consists of a sequence of nested cups on the domain word,
a sequence of nested caps on the codomain word, and lines between the remaining
symbols in the words without any crossings of lines, cups or caps.
The cups in the domain are directed from $-$ to $+$, the caps in the codomain are directed from $+$ to $-$, and the lines are directed either from a $-$ in the domain to a $-$ in the codomain or from a $+$ in the codomain to a $+$ in the domain.
Diagram \eqref{Duex1} is an example of a morphism from $--+++-+-$ to $+++---$; diagrams \eqref{Duex2}
and \eqref{Duex3} give other examples.
\begin{eqnarray}\label{Duex1}
\begin{aligned}
\psscalebox{0.8 0.8} % Change this value to rescale the drawing.
{
\begin{pspicture}(0,-1.9817871)(9.08,1.9817871)
\psbezier[linecolor=black, linewidth=0.04](0.9,1.6941504)(0.9,0.8941504)(2.3,0.8941504)(2.3,1.694150390625)
\psbezier[linecolor=black, linewidth=0.04](0.2,1.6941504)(0.2,0.094150394)(2.98,0.05415039)(2.98,1.654150390625)
\psbezier[linecolor=black, linewidth=0.04](7.0802226,-1.6154858)(7.068209,-0.815576)(5.668367,-0.83659947)(5.6803803,-1.636509254832838)
\psbezier[linecolor=black, linewidth=0.04](7.7801437,-1.604974)(7.756117,-0.0051544686)(4.9758296,-0.0069055757)(4.9998565,-1.6067251629228894)
\rput[bl](8.44,1.7141504){$-$}
\rput[bl](2.08,1.7141504){$+$}
\rput[bl](5.46,1.6941504){$-$}
\rput[bl](8.4,-1.8858496){$-$}
\rput[bl](7.54,-1.9058496){$-$}
\rput[bl](6.88,-1.9058496){$-$}
\rput[bl](0.66,1.7541504){$-$}
\rput[bl](0.0,1.7541504){$-$}
\rput[bl](2.76,1.7141504){$+$}
\rput[bl](3.86,1.6941504){$+$}
\rput[bl](3.84,-1.8658496){$+$}
\rput[bl](4.8,-1.8658496){$+$}
\rput[bl](5.4,-1.8858496){$+$}
\rput[bl](6.86,1.6741503){$+$}
\psline[linecolor=black, linewidth=0.04](4.04,1.6341504)(4.0,-1.5258496)
\psline[linecolor=black, linewidth=0.04](8.64,1.6341504)(8.62,-1.5258496)
\psbezier[linecolor=black, linewidth=0.04](5.64,1.6141504)(5.64,0.8141504)(7.04,0.8141504)(7.04,1.614150390625)
\end{pspicture}
}
\end{aligned}
\end{eqnarray}
Composition is performed by vertical stacking followed by employing the snake
equations to straighten out the snaking into cups, caps and lines. For example,
the composable morphisms $$(-++---+--++) \lra (+--+---++-+) \lra (+-++---)$$ as shown
stacked vertically in \eqref{Duex2} have composite as shown in \eqref{Duex3}.
\begin{eqnarray}\label{Duex2}
\begin{aligned}
\psscalebox{0.8 0.8} % Change this value to rescale the drawing.
{
\begin{pspicture}(0,-2.0517871)(10.02,2.0517871)
\rput[bl](0.0,1.8241504){$-$}
\rput[bl](0.72,1.8241504){$+$}
\rput[bl](2.42,1.8241504){$-$}
\rput[bl](3.16,1.8241504){$-$}
\rput[bl](3.88,1.8041503){$-$}
\rput[bl](5.42,1.8041503){$-$}
\rput[bl](6.18,1.8041503){$-$}
\rput[bl](2.38,-0.09584961){$-$}
\rput[bl](3.14,-0.09584961){$-$}
\rput[bl](4.7,-0.07584961){$-$}
\rput[bl](5.5,-0.07584961){$-$}
\rput[bl](6.28,-0.07584961){$-$}
\rput[bl](8.56,-0.07584961){$-$}
\rput[bl](2.26,-1.9558496){$-$}
\rput[bl](4.66,-1.9358497){$-$}
\rput[bl](5.42,-1.9358497){$-$}
\rput[bl](6.2,-1.9558496){$-$}
\rput[bl](1.56,1.8241504){$+$}
\rput[bl](4.56,1.8041503){$+$}
\rput[bl](7.02,1.8241504){$+$}
\rput[bl](9.16,1.8241504){$+$}
\rput[bl](1.54,-0.07584961){$+$}
\rput[bl](3.94,-0.07584961){$+$}
\rput[bl](7.0,-0.07584961){$+$}
\rput[bl](7.86,-0.07584961){$+$}
\rput[bl](9.3,-0.055849608){$+$}
\rput[bl](1.48,-1.9758496){$+$}
\rput[bl](2.98,-1.9558496){$+$}
\rput[bl](3.78,-1.9558496){$+$}
\psbezier[linecolor=black, linewidth=0.04](0.22,1.7241504)(0.22,0.9241504)(0.9,0.9241504)(0.9,1.724150390625)
\psline[linecolor=black, linewidth=0.04](1.78,1.7841504)(1.76,0.22415039)
\psline[linecolor=black, linewidth=0.04](2.66,1.7841504)(2.6,0.1841504)
\psline[linecolor=black, linewidth=0.04](3.36,1.7641504)(3.32,0.24415039)
\psbezier[linecolor=black, linewidth=0.04](4.12,1.7441504)(4.12,0.9441504)(4.8,0.9441504)(4.8,1.744150390625)
\psbezier[linecolor=black, linewidth=0.04](3.38,-0.13584961)(3.38,-0.9358496)(4.18,-0.9558496)(4.18,-0.155849609375)
\psbezier[linecolor=black, linewidth=0.04](8.759755,0.23706034)(8.729379,1.0364834)(8.049869,1.0106635)(8.080245,0.21124044211746423)
\psbezier[linecolor=black, linewidth=0.04](8.8,-0.07584961)(8.8,-0.8758496)(9.48,-0.8758496)(9.48,-0.075849609375)
\psbezier[linecolor=black, linewidth=0.04](5.7,-0.11584961)(5.7,-0.9158496)(8.04,-0.9558496)(8.04,-0.155849609375)
\psbezier[linecolor=black, linewidth=0.04](6.56,-0.09584961)(6.56,-0.71584964)(7.36,-0.4358496)(7.24,-0.095849609375)
\psbezier[linecolor=black, linewidth=0.04](4.939755,0.21706034)(4.9093785,1.0164834)(4.189869,1.0306635)(4.2202454,0.23124044211746878)
\psbezier[linecolor=black, linewidth=0.04](4.939755,-1.6229397)(4.979841,-1.2035166)(3.9520762,-1.1493365)(4.020245,-1.6487595578825311)
\psbezier[linecolor=black, linewidth=0.04](5.6797547,-1.6629397)(5.6493783,-0.86351657)(3.189869,-0.88933647)(3.2202451,-1.6887595578825358)
\psline[linecolor=black, linewidth=0.04](1.8,-0.17584962)(1.74,-1.6558496)
\psline[linecolor=black, linewidth=0.04](2.58,-0.13584961)(2.48,-1.6358496)
\psline[linecolor=black, linewidth=0.04](4.98,-0.1558496)(6.34,-1.5958496)
\psline[linecolor=black, linewidth=0.04](5.62,1.7641504)(5.64,0.26415038)
\psline[linecolor=black, linewidth=0.04](6.4,1.7641504)(6.48,0.26415038)
\psline[linecolor=black, linewidth=0.04](7.22,1.7441504)(7.24,0.2041504)
\psline[linecolor=black, linewidth=0.04](9.4,1.7241504)(9.44,0.3041504)
\end{pspicture}
}
\end{aligned}
\end{eqnarray}

\begin{eqnarray}\label{Duex3}
\begin{aligned}
\psscalebox{0.8 0.8} % Change this value to rescale the drawing.
{
\begin{pspicture}(0,-1.3817871)(8.5,1.3817871)
\rput[bl](0.0,1.1541504){$-$}
\rput[bl](0.72,1.1541504){$+$}
\rput[bl](2.42,1.1541504){$-$}
\rput[bl](3.16,1.1541504){$-$}
\rput[bl](3.88,1.1341504){$-$}
\rput[bl](5.42,1.1341504){$-$}
\rput[bl](6.18,1.1341504){$-$}
\rput[bl](2.42,-1.2858496){$-$}
\rput[bl](4.82,-1.2658496){$-$}
\rput[bl](5.58,-1.2658496){$-$}
\rput[bl](6.36,-1.2858496){$-$}
\rput[bl](1.56,1.1541504){$+$}
\rput[bl](4.56,1.1341504){$+$}
\rput[bl](6.92,1.1341504){$+$}
\rput[bl](7.78,1.1341504){$+$}
\rput[bl](1.64,-1.3058496){$+$}
\rput[bl](3.14,-1.2858496){$+$}
\rput[bl](3.94,-1.2858496){$+$}
\psbezier[linecolor=black, linewidth=0.04](0.22,1.0541503)(0.22,0.2541504)(0.9,0.2541504)(0.9,1.054150390625)
\psline[linecolor=black, linewidth=0.04](1.76,1.0741504)(1.86,-0.96584964)
\psline[linecolor=black, linewidth=0.04](2.64,1.0941504)(2.64,-1.0058496)
\psline[linecolor=black, linewidth=0.04](3.36,1.0941504)(6.58,-0.96584964)
\psbezier[linecolor=black, linewidth=0.04](4.12,1.0741504)(4.3,0.71415037)(4.58,0.6341504)(4.8,1.074150390625)
\psbezier[linecolor=black, linewidth=0.04](5.039755,-0.99293965)(5.0798416,-0.57351655)(4.0520763,-0.51933646)(4.120245,-1.0187595578825313)
\psbezier[linecolor=black, linewidth=0.04](5.859755,-0.97293967)(5.8293786,-0.17351657)(3.3698688,-0.19933647)(3.4002452,-0.9987595578825358)
\psbezier[linecolor=black, linewidth=0.04](6.46,1.0541503)(6.64,0.6941504)(6.92,0.6141504)(7.14,1.054150390625)
\psbezier[linecolor=black, linewidth=0.04](5.66,1.0541503)(5.66,0.2541504)(8.0,0.2541504)(8.0,1.054150390625)
\end{pspicture}
}
\end{aligned}
\end{eqnarray}
\begin{remark}
The string diagrams we have described are in ``normal form''. 
Composition is done by vertical stacking and then moving to normal form.
Readers requiring a geometric model which accommodates the diagrams arising from
vertical stacking of the normal diagrams {\em per se} should see \cite{xii}. 
Winding numbers are involved.
\end{remark}
\begin{remark}
These string diagrams and their composition are much like those of the Temperley-Lieb
algebra \cite{TempLieb} except that their directed nature prohibits the creation of loops. 
They are also particularly non-tangled tangles as occurring in \cite{Yet1988, Shum1994}.
\end{remark}

\section{Combinatorial definition of $\mathrm{Dpr}$}\label{CdDpr}
A subset $K$ of an ordered set $M$ is an {\em interval} when 
$a\le b\le c$ and $a,c\in K$ imply $b\in K$.
For $a, c\in M$, we have the interval $[a,c] = \{b \in M : a\le b\le c \}$. 

Let $S$ be a subset of a linearly ordered set $M$ and let $K$ be a finite interval in $M$.
We say that $K$ is an {\em $S$-cup} [respectively, {\em $S$-cap}] when $K$ has even
cardinality and $S\cap K$ is a final [respectively, initial] segment of $K$ containing exactly half of the elements of $K$. Notice that, if $K$ and $L$ are both $S$-cups or both $S$-caps and 
$K\cap L\neq \varnothing$, then either $K\subseteq L$ or $L\subseteq K$. 

Suppose $M$ and $S$ are as above and $a\in M$. We write $a^{\cup}$ for the element of $M$ 
for which the interval $[a,a^{\cup}]$ in $M$ is an $S$-cup. 
Such an element $a^{\cup}$ may not exist but it will if $M$ is a union of $S$-cups and $a\notin S$.
Similarly, we write $a^{\cap}$ for the element of $M$ 
for which the interval $[a,a^{\cap}]$ in $M$ is an $S$-cap. 

\begin{Example} If $M=\{m_1<m_2<m_3<m_4\}$, $S = \{m_3, m_4\}$ then $m_1^{\cup}=m_4$ and $m_2^{\cup}=m_3$.  
\end{Example}

Suppose $H\subseteq M$ and $x, x'\in M$. A sequence 
$$x=x_0, x_1, \dots, x_{n}=x'$$ is said to be {\em snaking in} $H$ out of $x$ into $x'$ when 
 $x_i\in H$ for $0 < i < n$ and either $x_{m+1}= x_{m}^{\cup}$ or $x_{m+1}= x_{m}^{\cap}$ for $0 \le m < n$. We write $x\rightsquigarrow_H x'$ when such a sequence exists.

If $M$ and $N$ are linearly ordered sets and $A\subseteq M$, $B\subseteq N$ are subsets
of the same cardinality $\#A=\#B$ then there is a unique order-preserving bijection $A\to B$;
we call $a\in A$ and $b\in B$ {\em cobbers} when they correspond under this bijection.     

We now describe a monoidal category $\mathrm{Dpr}$ which is equivalent to the geometric version of Section~\ref{gmDpr}. 
An object $(M,S)$ consists of a finite
linearly ordered set $M$ and a subset $S\subseteq M$. 
Of course, we can also think of such an object as a function 
$\chi_S : M\to \{-.+\}$ for which $S$ is the inverse image of $+$. 

A morphism $(A,B) : (M,S) \to (N,T)$ in $\mathrm{Dpr}$ consists of subsets $A\subseteq M$ and $B\subseteq N$ such that 
\begin{itemize}
\item[(i)] $\#A=\#B$,
\item[(ii)] for all cobbers $a\in A$ and $b\in B$, $a\in S$ if and only if $b\in T$,
\item[(iii)] $A' = M\backslash A$ is a union of $S$-cup intervals in $M$,
\item[(iv)] $B'=N\backslash B$ is a union of $T$-cap intervals in $N$.

\end{itemize}
The identity morphism of $(M,S)$ is $(M,M) : (M,S)\to (M,S)$. 

The composite 
$$(M,S) \xra{(E,F)} (P,U) = ((M,S) \xra{(A,B)} (N,T) \xra{(C,D)} (P,U))$$ is defined by
\begin{eqnarray*}
E = \{x \in A : \ \text{the cobber} \ y\in B \ \text{of} \ x \ \text{has} \ y\rightsquigarrow_{B'\cap C'}y' \ \text{with} \ y'\in C \} \\ 
\ \cup \ \{x' \in A : \ \text{the cobber} \ y'\in B \ \text{of} \ x' \ \text{has} \ y\rightsquigarrow_{B'\cap C'} y' \ \text{with} \ y\in C \} 
\end{eqnarray*}
and
\begin{eqnarray*}
F = \{z \in D : \ \text{the cobber} \ y\in C \ \text{of} \ z \ \text{has} \ y\rightsquigarrow_{B'\cap C'}y'  \ \text{with} \ y'\in B \} \\ 
\ \cup \ \{z' \in D : \ \text{the cobber} \ y'\in C \ \text{of} \ z' \ \text{has} \ y\rightsquigarrow_{B'\cap C'}y' \ \text{with} \ y\in B \} \ . 
\end{eqnarray*}

The intersection of the two sets whose union gives the definition of $E$ contains precisely those $x\in A$
whose cobber $y\in B$ is also in $C$.
The intersection of the two sets whose union is $F$ contains precisely those $z\in D$
whose cobber $y\in C$ is also in $B$.
The sets $E$ and $F$ have the same cardinality: the cobber in $F$ of an $x\in E$ in the first set of the union is the cobber in $D$ of the $y'\in C$ while the cobber in $F$ of an $x'\in E$ in the second set of the union is the cobber in $D$ of the $y\in C$. 

The complement $E'$ of $E$ is a union of $S$-cups; indeed, we have the disjoint union
\begin{eqnarray*}
E' = A'\cup \{x, x' \in A : \ \text{the cobbers} \ y, y'\in B \ \text{of} \ x, x' \ \text{have} \ y\rightsquigarrow_{B'\cap C'} y' \ \}  \  . 
\end{eqnarray*}  
Similarly,
\begin{eqnarray*}
F' = D'\cup \{z, z' \in D : \ \text{the cobbers} \ y, y'\in C \ \text{of} \ z, z' \ \text{have} \ y\rightsquigarrow_{B'\cap C'} y' \ \}  \  . 
\end{eqnarray*} 

\begin{proposition}\label{invmorinDu}
A morphism $(A,B) : (M,S) \to (N,T)$ in $\mathrm{Dpr}$ is invertible if and only if $A=M$, $B=N$.
 \end{proposition}

The tensor product for $\mathrm{Dpr}$ is defined by componentwise ordinal sum which we denote by $+$.
So 
$$(M,S)\otimes (M_1,S_1) = (M+M_1, S+S_1) \ \text{and} \ (A,B)\otimes (A_1,B_1) = (A+A_1, B+B_1) \ .$$  
We write $\mathbf{n}$ for the ordinal $\{0, 1, \dots, n-1\}$.
The unit for this tensor product is $\mathbb{I} = (\mathbf{0}, \varnothing)$.  

We shall now distinguish a dual pair in $\mathrm{Dpr}$. 
We have two objects $- = (\mathbf{1},\varnothing)$
and $+ = (\mathbf{1},\{0\})$ of $\mathrm{Dpr}$. 
Then $-\ox + = (\mathbf{2},\{1\})$ and $+\ox - = (\mathbf{2},\{0\})$. 
Put $\varepsilon = (\varnothing, \varnothing) : -\ox + \to \mathbb{I}$
and $\eta = (\varnothing, \varnothing) : \mathbb{I}\to +\ox -$.
It is an easy exercise in the definition of composition in $\mathrm{Dpr}$ to
see that the composites \eqref{radj} (with $X=-$ and $Y=+$) are identities.
So $-\dashv +$.
 
Every object of $\mathrm{Dpr}$ is uniquely isomorphic to one of the form $(\mathbf{n}, S)$.
For every finite linearly ordered set $M$, we have $(M, \varnothing)\dashv (M,M)$
in $\mathrm{Dpr}$ by tensoring $-\dashv +$ $\#M$ times. 
 
Let us call {\em elementary} those morphisms in $\mathrm{Dpr}$ of the form 
$$(\varnothing,\varnothing) : (M,S) \to (N,T)$$
where $M$ and $N$ are not both empty.

\begin{proposition}\label{decompofgeneralmorph}
Every morphism of $\mathrm{Dpr}$ is uniquely of the form 
$$f_1\ox e_1\ox f_2 \ox e_2\ox \dots \ox e_{k-1}\ox f_k$$
where $f_1, \dots , f_k$ are invertible morphisms and 
$e_1, \dots , e_{k-1}$ are elementary morphisms.
\end{proposition}
\begin{proof}
Take any morphism $(A,B) : (M,S)\to (N,T)$ of $\mathrm{Dpr}$.
Then we can write $A = K_1+K_2+ \dots K_k$ and $B = L_1+L_2+ \dots L_k$
where the $K_i$ are intervals in $M$ which are maximal as subsets of $A$, the $L_j$ are intervals in $N$ which are maximal as subsets of $B$, and
if $a\in K_i$ in $A$ has cobber $b\in L_j$ in $B$ then $i=j$. 
For $1\le i < k$, let $E_i\subseteq A'$ consist of all the elements between $K_i$ and $K_{i+1}$
and let $F_i\subseteq B'$ consist of all the elements between $L_i$ and $L_{i+1}$;
not both are empty or else $K_i$ and $K_{i+1}$ or $L_i$ and $L_{i+1}$ would intersect.
Then we have $f_i : (K_i, S\cap K_i) \to (L_i, T\cap L_i)$ and $e_i : (E_i, S\cap E_i) \to (F_i, T\cap F_i)$ as desired.  
\end{proof}

For monoidal categories $\CA$ and $\CV$, write $\mathrm{StMon}_{\mathrm{g}}(\CA,\CV)$ 
for the groupoid of strong monoidal functors from $\CA$ to $\CV$ and monoidal natural isomorphisms. Write $\mathrm{Dp}_{\mathrm{g}}\CV$ for the subcategory of $\mathrm{Dp}\CV$ consisting of all the objects but only the invertible morphisms.

\begin{proposition}
For any monoidal category $\CV$, the functor
\begin{eqnarray*}
\mathrm{StMon}_{\mathrm{g}}(\mathrm{Dpr},\CV) \lra \mathrm{Dp}_{\mathrm{g}}\CV \ , \ \Phi \mapsto (\Phi -\dashv \Phi +)
\end{eqnarray*}
is an equivalence of groupoids.  
\end{proposition}
\begin{proof}
Take a dual pair $X\dashv Y$ in $\CV$.
Define $\mathrm{Dpr}\xra{\Phi}\CV$ on objects by
$$\Phi(M,S) = \bigotimes_{m\in M}Z_m$$
where $Z_m = X$ for $m\notin S$ and $Z_m = Y$ for $m\in S$.
Define $\Phi$ to take invertible morphisms 
(see Proposition~\ref{invmorinDu}) of $\mathrm{Dpr}$
to identity morphisms in $\CV$. 
In accord with the desire for $\Phi$ to be a strong monoidal functor, a nested cup $\mathbb{I} \to (M,S)$ with $\#M=2h$ must be taken to the counit
of $X^{\ox h}\dashv Y^{\ox h}$.
A nested cap $(N,T) \to \mathbb{I}$ with $\#N=2k$ must be taken to the unit
of $X^{\ox k}\dashv Y^{\ox k}$.
The effect of $\Phi$ on a general morphism is then forced by Proposition~\ref{decompofgeneralmorph}. The fact that $\Phi$ preserves composition
follows from the snake equations for the dual pairs $X^{\ox k}\dashv Y^{\ox k}$.
Clearly the functor of the Proposition takes this $\Phi$ to $X\dashv Y$. 
\end{proof}
\section{Relationship to the free adjunction}\label{Rfa}
The 2-category freely generated by an adjunction was first considered in \cite{Aud1974}
with an explicit model $\mathrm{Adj}$ described in \cite{28}.
The 2-category $\mathrm{Adj}$ has two objects which we will now denote by $\flat$ and $\sharp$. The hom category $\mathrm{Adj}(\flat, \flat)$ is the usual algebraist's $\Delta$:
the objects are the finite ordinals $\mathbf{n}=\{0,1,\cdots n-1\}$ and the morphisms order-preserving functions. 
The hom category $\mathrm{Adj}(\sharp, \sharp)$ is the category $\Delta_{\top,\bot}$:
the objects are the non-empty ordinals $\mathbf{n}$ and the morphisms order-preserving functions which preserve first and last elements. 
The hom category $\mathrm{Adj}(\flat, \sharp)$ is the category $\Delta_{\top}$:
the objects are the non-empty ordinals $\mathbf{n}$ and the morphisms order-preserving functions which preserve last elements.
The hom category $\mathrm{Adj}(\sharp, \flat)$ is the category $\Delta_{\bot}$:
the objects are the non-empty ordinals $\mathbf{n}$ and the morphisms order-preserving functions which preserve first elements.
Composition is described in \cite{28}; in particular, the composition functor
$$\mathrm{Adj}(\flat, \flat)\times \mathrm{Adj}(\flat, \flat) \lra \mathrm{Adj}(\flat, \flat)$$
is ordinal sum, the tensor product of the monoidal category $\Delta$; and composition
$$\mathrm{Adj}(\sharp, \sharp)\times \mathrm{Adj}(\sharp, \sharp) \lra \mathrm{Adj}(\sharp, \sharp)$$
is the result of transporting ordinal sum across the duality $\Delta^{\mathrm{op}} \simeq \Delta_{\top,\bot}, \ \mathbf{n}\mapsto \mathbf{n+1}$. 

For a monoidal category $\CA$, we write $\Sigma \CA$ for the one object bicategory
whose hom category is $\CA$ and whose composition is the tensor product of $\CA$.
In particular, we have the bicategory $\Sigma \mathrm{Dpr}$ which contains the adjunction $-\dashv +$. Therefore, there is a pseudofunctor 
\begin{eqnarray}\label{Theta}
\Theta : \mathrm{Adj} \lra \Sigma \mathrm{Dpr}
\end{eqnarray}
taking the generating adjunction in $\mathrm{Adj}$ to $-\dashv +$.   

In order to study the image of $\Theta$, we distinguish some objects of $\mathrm{Dpr}$.
We call $(M,S)$ {\em alternating} when, for all consecutive elements $a, b\in M$, we have
$a\in S$ if and only if $b\notin S$. Notice that any $S$-cup or $S$-cap in such an $M$
can only have cardinality $0$ or $2$.

Let us look at the effect of $\Theta$ on the endomorphism hom of $\flat \in \mathrm{Adj}$.
This is a strong monoidal functor $\Theta_{\flat, \flat} : \Delta \to \mathrm{Dpr}$.
The object $\mathbf{n}$ of $\Delta$ is taken to the alternating object 
$(\mathbf{2n}, \mathrm{Ev}_n)$ where $\mathrm{Ev}_n = \{0, 2, \dots, 2n-2\}$
consists of the even elements of $\mathbf{2n}$. 
The surjection $\sigma_i : \mathbf{n+1}\to \mathbf{n}$ in $\Delta$, which identifies $i$
and $i+1$, is taken to $(\mathrm{de}_i, \mathbf{2n+2}) : (\mathbf{2n+2}, \mathrm{Ev}_{n+1})\to (\mathbf{2n}, \mathrm{Ev}_{n})$ where $\mathrm{de}_i$ is obtained from $\mathbf{2n+2}$ by deleting $2i+1$ and $2i+2$.
The injection $\partial_i : \mathbf{n}\to \mathbf{n+1}$ in $\Delta$, which does not have $i$
in its image, is taken to $(\mathbf{2n}, \mathrm{fa}_i) : (\mathbf{2n}, \mathrm{Ev}_n)\to (\mathbf{2n+2}, \mathrm{Ev}_{n+1})$ where $\mathrm{fa}_i$ is obtained from $\mathbf{2n+2}$ by deleting $2i$ and $2i+1$.     
\begin{eqnarray}\label{Deltagens}
\begin{aligned}
\psscalebox{0.8 0.8} % Change this value to rescale the drawing.
{
\begin{pspicture}(0,-1.5469922)(16.52,1.5469922)
\rput[bl](7.38,-1.4110547){$\partial_2 : \mathbf{4}\longrightarrow \mathbf{5}$}
\rput[bl](1.74,-1.4710547){$\sigma_2 : \mathbf{5}\longrightarrow \mathbf{4}$}
\rput[bl](6.94,1.2889453){$+ \ - \ + \ - \ + \ - \ + \ -$}
\rput[bl](6.64,-0.8910547){$+ \ - \ + \ - \ + \ - \ + \ - \ + \ -$}
\rput[bl](0.86,-0.8910547){$+ \ - \ + \ - \ + \ - \ + \ -$}
\rput[bl](0.0,1.2889453){$+ \ - \ + \ - \ + \ - \ + \ - \ + \ -$}
\psbezier[linecolor=black, linewidth=0.04](2.7,1.1489453)(2.72,0.5089453)(3.24,0.5289453)(3.22,1.1889453125)
\psbezier[linecolor=black, linewidth=0.04](9.420068,-0.53447723)(9.371908,0.040234484)(8.831322,0.047079634)(8.839932,-0.5676321124760306)
\psline[linecolor=black, linewidth=0.04](7.04,1.1689453)(6.78,-0.5510547)
\psline[linecolor=black, linewidth=0.04](7.56,1.2289453)(7.32,-0.5910547)
\psline[linecolor=black, linewidth=0.04](8.12,1.1889453)(7.84,-0.5710547)
\psline[linecolor=black, linewidth=0.04](8.56,1.1889453)(8.38,-0.61105466)
\psline[linecolor=black, linewidth=0.04](9.22,1.1689453)(9.96,-0.61105466)
\psline[linecolor=black, linewidth=0.04](9.72,1.1689453)(10.46,-0.5910547)
\psline[linecolor=black, linewidth=0.04](10.24,1.1289454)(10.96,-0.5910547)
\psline[linecolor=black, linewidth=0.04](10.76,1.1689453)(11.34,-0.6510547)
\psline[linecolor=black, linewidth=0.04](0.28,1.2089453)(1.02,-0.5310547)
\psline[linecolor=black, linewidth=0.04](0.78,1.1689453)(1.52,-0.5710547)
\psline[linecolor=black, linewidth=0.04](1.32,1.1889453)(2.06,-0.5510547)
\psline[linecolor=black, linewidth=0.04](1.82,1.1689453)(2.56,-0.5710547)
\psline[linecolor=black, linewidth=0.04](2.28,1.1689453)(3.02,-0.5710547)
\psline[linecolor=black, linewidth=0.04](3.74,1.1689453)(3.58,-0.5310547)
\psline[linecolor=black, linewidth=0.04](4.26,1.1489453)(4.06,-0.5510547)
\psline[linecolor=black, linewidth=0.04](4.66,1.1689453)(4.52,-0.5310547)
\end{pspicture}
}
\end{aligned}
\end{eqnarray}
We define $\Theta$ on general morphisms of $\Delta$ using the following notation from \cite{118}. For each morphism $\xi : \mathbf{m}\to \mathbf{n}$ in $\Delta$, we put
$$\xi^{\ell} = \{i\in \mathbf{m-1} : \xi(i)=\xi(i+1)\} \ \text{  and  } \ \xi^{r} = \{j \in \mathbf{n} : j\notin \mathrm{im}\xi \} \ .$$
Notice that $\xi^{\ell} = \varnothing$ means that $\xi$ is injective and $\xi^{r} = \varnothing$ means that $\xi$ is surjective.
In general, $\xi$ is determined by $(\xi^{\ell}, \xi^{r})$.
We have that $\Theta(\xi) = (A_{\xi}, B_{\xi}) : (\mathbf{2m}, \mathrm{Ev}_m) \to (\mathbf{2n}, \mathrm{Ev}_n)$ is defined by
$$A'_{\xi} = \{2i+1,2i+2 : i\in \xi^{\ell} \} \ \text{  and  } \ B'_{\xi} = \{2j, 2j+1 : j\in \xi^{r} \} \ ,$$
where, as before, the primed sets denote the appropriate complements.
The restricted possibility for cups and caps in alternating objects means that 
every morphism $(\mathbf{2m}, \mathrm{Ev}_m) \to (\mathbf{2n}, \mathrm{Ev}_n)$ 
must be of the form $(A_{\xi}, B_{\xi})$ for some $\xi : \mathbf{m} \to \mathbf{n}$ in $\Delta$. 
So the functor $\Theta_{\flat, \flat} : \Delta \to \mathrm{Dpr}$ is fully faithful.
A similar analysis applies to the other three hom categories of $\mathrm{Adj}$ yielding:    
\begin{proposition}
The pseudofunctor $\Theta$ \eqref{Theta} is locally fully faithful.
\end{proposition}

We can say this differently if $\mathrm{Dpr}$ is given and $\mathrm{Adj}$ is to be obtained. 
Take the free category $\CJ$ on the directed graph
$$
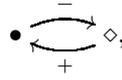
\xymatrix{
\bullet\ar@/^/[r]^{-}&\diamond\ar@/^/[l]^{+},
}$$
regarded as a locally discrete 2-category. 
Then $\mathrm{Adj}$ is obtained by factoring the 2-functor $\CJ \to \Sigma\mathrm{Dpr}$, which takes
the morphisms $-$ to $-$ and $+$ to $+$, into a 2-functor $\CJ \to \mathrm{Adj}$ which is bijective on both objects and 1-morphisms, and a functor $\mathrm{Adj} \to \Sigma\mathrm{Dpr}$ which is locally fully faithful. 

\section{Interlude on iterated duals}\label{Ioid}

Let $\Lambda$ denote a set equipped with a partial endofunction $\sigma : \Lambda \rightharpoonup \Lambda$
such that
\begin{itemize}
\item[a.] $\sigma{s} = \sigma{t}$ implies $s=t$;
\item[b.] $\sigma^n{s} = s$ for some $s$ implies $n=0$.
\end{itemize}
Let $\Lambda^*$ be the set of words $\underline{s} = s_1s_2\dots s_k$ in the alphabet $\Lambda$. 
Let 
$$|-| : \Lambda^* \to \mathbb{N}, \ \underline{s} \mapsto k$$
take a word to its length. Put $\langle k\rangle = \{1,2,\dots ,k\}$.

An interval $I$ in $\langle k\rangle$ is an {\em $\underline{s}$-cup} [respectively, {\em $\underline{s}$-cap}] when it has even cardinality and $s_n = \sigma s_m$ [respectively, $s_m = \sigma s_n$] whenever $m< n$ are both in $I$ and 
$$\#\{x\in I : x\le m\} = \#\{x\in I : n\le x\} \ .$$ 
We put $m^{\cup}=n$ [respectively, $m^{\cap}=n$] in this situation.
We will use the definition of {\em snaking} and {\em cobbers} as in Section~\ref{CdDpr}.

The goal in this section is to define the free monoidal category $\mathrm{D}(\Lambda,\sigma)$ containing the elements of $\Lambda$ as objects in such a way that $s \dashv \sigma{s}$ for all $s\in \Lambda$ on which $\sigma$ is defined.
First we give some examples.
\begin{Example}
\begin{itemize}
\item[0.] $\Lambda$ arbitrary and $\sigma$ with empty domain. Then $\mathrm{D}(\Lambda,\sigma)$
is the free monoidal category (discrete $\Lambda^*$) on the set $\Lambda$. 
\item[1.] $\Lambda = \{-,+\}$, $\sigma{-}=+$ and $\sigma{+}$ undefined. Then $\mathrm{D}(\Lambda,\sigma) = \mathrm{Dpr}$ as in Section~\ref{CdDpr}.
\item[2.] $\Lambda = \mathbb{N}$ and $\sigma{n}=n+1$ for all natural numbers $n$.
Then $\mathrm{D}(\Lambda,\sigma)$ is the free monoidal category $\mathrm{Dseq}$ containing a duality sequence $X_0\dashv X_1\dashv X_2 \dashv \dots$ of objects.
Equally, $\mathrm{D}(\Lambda,\sigma)$ is the free right autonomous monoidal category $\mathrm{Dseq}$ generated by a single object. A geometric model can be derived as in Section~\ref{gmDpr}.
The objects are words of natural numbers and composition leads to diagrams of the form shown in 
\eqref{Nex} for morphisms $$2\to 2\ 4\ 3\ 1\ 0\ 2\ 3\ 1\ 2 \to 2\ 4\ 3.$$
\item[3.] $\Lambda = \mathbb{Z}$ and $\sigma{n}=n+1$ for all integers $n$. Then $\mathrm{D}(\Lambda,\sigma)$ is the free autonomous monoidal category on a single generating object.
Equally, $\mathrm{D}(\Lambda,\sigma)$ is the free monoidal category $\mathrm{Dseq}$ containing a doubly infinite string of object dualities $\dots \dashv X_{-2}\dashv X_{-1}\dashv X_0\dashv X_1\dashv X_2 \dashv \dots$.
\item[4.] Consider an arbitrary set $\Lambda$ with an element $J\in \Lambda$, such that $\sigma$ is only defined on $J$, and set $\Pi:=\Lambda\backslash\{\sigma J\}$. Then $\mathrm{D}(\Lambda,\sigma)$ is $\CC[J^\vee]$ as defined in Section~\ref{ConstructionC[J*]} for $\CC:=\Pi^\ast$.
\end{itemize}
\end{Example}

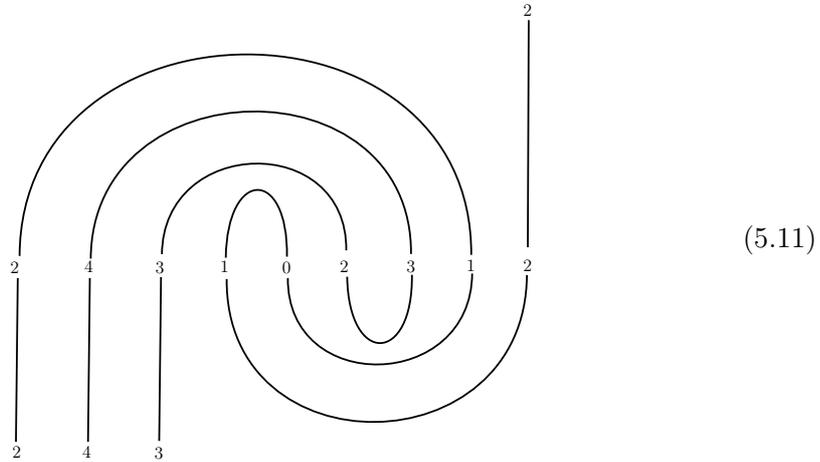
\begin{eqnarray}\label{Nex}
\begin{aligned}
\psscalebox{0.6 0.6} % Change this value to rescale the drawing.
{
\begin{pspicture}(0,-5.03333)(11.46,5.03333)
\rput[bl](0.0,-0.85739255){2}
\rput[bl](1.62,-0.83739257){4}
\rput[bl](3.18,-0.85739255){3}
\rput[bl](4.6,-0.83739257){1}
\rput[bl](5.96,-0.85739255){0}
\rput[bl](7.22,-0.83739257){2}
\rput[bl](8.68,-0.83739257){3}
\rput[bl](10.0,-0.7973926){1}
\rput[bl](11.24,-0.7973926){2}
\psbezier[linecolor=black, linewidth=0.04](10.100227,-0.44583118)(10.129241,5.4771996)(0.22878525,5.454077)(0.19977221,-0.46895396566861564)
\psbezier[linecolor=black, linewidth=0.04](8.780228,-0.4258312)(8.809241,3.841829)(1.7887852,3.738706)(1.7597722,-0.5289539656686156)
\psbezier[linecolor=black, linewidth=0.04](7.3602276,-0.3458312)(7.3892407,2.2664585)(3.3487852,2.1833355)(3.3197722,-0.42895396566861566)
\psbezier[linecolor=black, linewidth=0.04](6.060228,-0.4858312)(6.089241,1.4979208)(4.748785,1.474798)(4.7197723,-0.5089539656686156)
\psbezier[linecolor=black, linewidth=0.04](7.3803983,-0.92857707)(7.409493,-2.9123278)(8.828696,-2.8699589)(8.799602,-0.886208104680918)
\psbezier[linecolor=black, linewidth=0.04](6.079961,-0.9572844)(6.0617204,-3.5696714)(10.101799,-3.4698877)(10.120039,-0.8575007785186654)
\psbezier[linecolor=black, linewidth=0.04](4.7396855,-0.9813237)(4.70585,-5.2489486)(11.286479,-5.161086)(11.320314,-0.893461459904289)
\psline[linecolor=black, linewidth=0.04](0.16,-0.9573926)(0.12,-4.5773926)
\psline[linecolor=black, linewidth=0.04](1.74,-0.9573926)(1.7,-4.5773926)
\psline[linecolor=black, linewidth=0.04](3.3,-0.9373926)(3.26,-4.5573926)
\psline[linecolor=black, linewidth=0.04](11.36,4.7426076)(11.34,-0.27739257)
\rput[bl](0.04,-4.9373927){2}
\rput[bl](1.58,-4.9373927){4}
\rput[bl](3.16,-4.9573927){3}
\rput[bl](11.24,4.8426075){2}
\end{pspicture}
}
\end{aligned}
\end{eqnarray}

We now describe combinatorially the monoidal category $\mathrm{D}(\Lambda,\sigma)$ in general. 

The objects are the elements of $\Lambda^*$.  

A morphism $(A,B) : \underline{s} \to \underline{t}$ in $\mathrm{D}(\Lambda,\sigma)$ consists of subsets $A\subseteq \langle |\underline{s}|\rangle$ and $B\subseteq \langle |\underline{t}|\rangle$ such that 
\begin{itemize}
\item[(i)] $\#A=\#B$,
\item[(ii)] for all cobbers $a\in A$ and $b\in B$, $s_a = t_b$,
\item[(iii)] the complement $A' = \langle |\underline{s}|\rangle\backslash A$ is a union of $\underline{s}$-cups in $\langle |\underline{s}|\rangle$,
\item[(iv)] the complement $B'=\langle |\underline{t}|\rangle\backslash B$ is a union of $\underline{t}$-caps in $\langle |\underline{t}|\rangle$.
\end{itemize}
The identity morphism of $\underline{s}$ is $(\langle |\underline{s}|\rangle,\langle |\underline{s}|\rangle) : \underline{s}\to \underline{s}$. 

The composite 
$$\big{(}\langle |\underline{s}|\rangle \xra{(E,F)} \langle |\underline{u}|\rangle\big{)} = \big{(}\langle |\underline{s}|\rangle \xra{(A,B)} \langle |\underline{t}|\rangle \xra{(C,D)} \langle |\underline{u}|\rangle\big{)}$$ is defined by
\begin{eqnarray*}
E = \{x \in A : \ \text{the cobber} \ y\in B \ \text{of} \ x \ \text{has} \ y\rightsquigarrow_{B'\cap C'}y' \ \text{with} \ y'\in C \} \\ 
\ \cup \ \{x' \in A : \ \text{the cobber} \ y'\in B \ \text{of} \ x' \ \text{has} \ y\rightsquigarrow_{B'\cap C'} y' \ \text{with} \ y\in C \} 
\end{eqnarray*}
and
\begin{eqnarray*}
F = \{z \in D : \ \text{the cobber} \ y\in C \ \text{of} \ z \ \text{has} \ y\rightsquigarrow_{B'\cap C'}y'  \ \text{with} \ y'\in B \} \\ 
\ \cup \ \{z' \in D : \ \text{the cobber} \ y'\in C \ \text{of} \ z' \ \text{has} \ y\rightsquigarrow_{B'\cap C'}y' \ \text{with} \ y\in B \} \ . 
\end{eqnarray*}

The intersection of the two sets whose union gives the definition of $E$ contains precisely those $x\in A$
whose cobber $y\in B$ is also in $C$.
The intersection of the two sets whose union is $F$ contains precisely those $z\in D$
whose cobber $y\in C$ is also in $B$.
The sets $E$ and $F$ have the same cardinality: the cobber in $F$ of an $x\in E$ in the first set of the union is the cobber in $D$ of the $y'\in C$ while the cobber in $F$ of an $x'\in E$ in the second set of the union is the cobber in $D$ of the $y\in C$. 

The complement $E'$ of $E$ is a union of $\underline{s}$-cups; indeed, we have the disjoint union
\begin{eqnarray*}
E' = A'\cup \{x, x' \in A : \ \text{the cobbers} \ y, y'\in B \ \text{of} \ x, x' \ \text{have} \ y\rightsquigarrow_{B'\cap C'} y' \ \}  \  . 
\end{eqnarray*}  
Similarly,
\begin{eqnarray*}
F' = D'\cup \{z, z' \in D : \ \text{the cobbers} \ y, y'\in C \ \text{of} \ z, z' \ \text{have} \ y\rightsquigarrow_{B'\cap C'} y' \ \}  \  . 
\end{eqnarray*} 

The tensor product for $\mathrm{D}(\Lambda,\sigma)$ is defined by juxtaposition
$\underline{s}\ox \underline{u} = \underline{s} \underline{u}$ on objects.
Note that $|\underline{s}\underline{u}| = |\underline{s}| + |\underline{u}|$ and so $\langle |\underline{s}\underline{u}|\rangle$ is the ordinal sum $\langle |\underline{s}|\rangle + \langle |\underline{u}|\rangle$. 
So 
$$(\underline{s}\xra{(A,B)}\underline{t})\otimes (\underline{u}\xra{(C,D)}\underline{v}) = (\underline{s} \underline{u}\xra{(A+C, B+D)}\underline{t} \underline{v}) \ .$$  

The unit $\mathbb{I}$ for this tensor product is of course the word of length $0$.  

For each $a\in \Lambda$ at which $\sigma$ is defined we have a duality $a\dashv \sigma a$ in $\mathrm{D}(\Lambda,\sigma)$ where we are identifying elements of $\Lambda$ with elements of $\Lambda^*$ as the one letter words. 
Put $\varepsilon = (\varnothing, \varnothing) :  a\ox \sigma a \to \mathbb{I}$
and $\eta = (\varnothing, \varnothing) : \mathbb{I}\to \sigma a\ox a$.
It is an easy exercise in the definition of composition in $\mathrm{D}(\Lambda,\sigma)$ to
see that the composites \eqref{radj} (with $X=a$ and $Y=\sigma a$) are identities.
So indeed $a\dashv \sigma a$.

\begin{proposition}\label{invmorinDLambda}
The only invertible morphisms in $\mathrm{D}(\Lambda,\sigma)$ are identities. 
 \end{proposition}
 
 As in Section~\ref{CdDpr}, we call {\em elementary} those morphisms in $\mathrm{D}(\Lambda,\sigma)$ of the form 
$$(\varnothing,\varnothing) : \underline{s} \to \underline{u}$$
where $|\underline{s}|$ and $|\underline{t}|$ are not both $0$.
Proposition~\ref{decompofgeneralmorph} generalizes.

\begin{proposition}\label{gendecompofgeneralmorph}
Every morphism of $\mathrm{D}(\Lambda,\sigma)$ is uniquely of the form 
$$f_1\ox e_1\ox f_2 \ox e_2\ox \dots \ox e_{k-1}\ox f_k$$
where $f_1, \dots , f_k$ are identity morphisms and 
$e_1, \dots , e_{k-1}$ are elementary morphisms.
\end{proposition}

 \begin{proposition}
For any strict monoidal category $\CV$ and any family $(X_a)_{a\in \Lambda}$ of objects of $\CV$ 
with chosen dualities $X_a\dashv X_{\sigma a}$, there exists a unique strict monoidal functor 
$\Phi : \mathrm{D}(\Lambda,\sigma) \lra \CV$ such that $\Phi a = X_a$ for all $a\in \Lambda$
and $\Phi$ maps the dualities $a\dashv \sigma a$ to the chosen dualities $X_a\dashv X_{\sigma a}$.
\end{proposition}

\begin{remark}
If $a\in \Lambda$ is in the domain of $\sigma$ then the monoidal full subcategory of $\mathrm{D}(\Lambda,\sigma)$, whose objects are the words in letters $a$ and $\sigma{a}$, is equivalent
to $\mathrm{Dpr}$.  
\end{remark}

\section{Construction of $\CC[J^{\vee}]$}\label{ConstructionC[J*]}

We write $[m]$ for the set $\{0, 1, \dots , m\}$ and $\partial_i : [m-1] \to [m]$ for the order-preserving injective function whose image does not contain $i$. 

Let $J$ be an object of the strict monoidal category $\CC$.

At this stage we only discuss the objects of the desired strict monoidal category $\CC[J^{\vee}]$ and their tensor product, by introducing the corresponding monoid.
Let $\CG_0$ be the set of functions $U : [m]\to \ob\CC$. Objects of $\CC$ are identified with functions $[0]\to \ob\CC$. For another $V : [n]\to \ob\CC$, the `tensor product' $U\ox V : [m+n]\to \ob\CC$ is defined by
\begin{eqnarray}
(U\ox V)_i = \left\{
\begin{array}{ll}
U_i & \text{for } i < m \\
U_m\ox V_0 & \text{for } i = m \\
V_{i-m} & \text{for } i > m \ .
\end{array} \right.
\end{eqnarray}
We thus obtain a monoid $(\CG_0,\otimes,\mathbb{I})$ since
\begin{eqnarray*}
\mathbb{I}\ox U = U = U\ox \mathbb{I} \ , \ (U\ox V)\ox W = U\ox (V\ox W) \ .
\end{eqnarray*}
We will already refer to elements of $\CG_0$ as the `objects of $\CC[J^{\vee}]$'.

An object of $\CC[J^{\vee}]$ of special interest is 
\begin{eqnarray*}
J^{\vee} : [1] \to \ob\CC \ , \  J^{\vee}_0 = J^{\vee}_1 = \mathbb{I} \ .
\end{eqnarray*}
Then we see that each object $U$ of $\CC[J^{\vee}]$ is uniquely decomposable in the form
\begin{eqnarray}\label{oxdecomp}
U = U_0\ox J^{\vee}\ox U_1\ox J^{\vee}\ox \dots \ox J^{\vee}\ox U_m   
\end{eqnarray}
with $U_0, U_1, \dots, U_m \in \CC$.

Now we will describe a directed graph $\CG=(\CG_0,\CG_1)$.
We introduce two symbols 
\begin{eqnarray*}
J\ox J^{\vee}\xra{\varepsilon} \mathbb{I} \ \text{ and } \ \mathbb{I} \xra{\eta} J^{\vee}\ox J \ .
\end{eqnarray*}
For each morphism $f : C\to D$ of $\CC$, we introduce a symbol $\bar{f} : C\to D$. 
These symbols $\varepsilon$, $\eta$ and the $\bar{f}$ will be called the {\em primitive edges} of $\CG$. 
The general edges of $\CG$ are the whiskered primitive edges: that is, those of the form
\begin{eqnarray}\label{genedge}
U\ox A \ox V \xra{U\ox p\ox V} U\ox B \ox V
\end{eqnarray}
where $U, V\in \CG_0$, $A\xra{p}B$ is primitive, and the tensor in $U\ox p\ox V$ is formal. 
For a general edge \eqref{genedge} and $T, W\in \CG_0$, define 
$$T\ox (U\ox p\ox V) \ox W = (T\ox U)\ox p\ox (V \ox W)$$ 
where $T\ox U$ and $V\ox W$ are (tensor) products in $\CG_0$. 

In the free category $\CF\CG$ on the graph $\CG$, if $\ell = a_k\dots a_1$ with $a_i\in\CG_1$, 
we put $U\ox \ell \ox V = (U\ox a_k\ox V)\dots (U\ox a_1\ox V)$. 

We introduce the following relations on morphisms of $\CF\CG$, where $C,D,E,F\in\CC$ and $U,V,W,X,Y,Z\in\CG_0$:
\begin{itemize}
\item[(i)] for $f\in \CC (C,D)$,
$$(E\ox C \ox F \xra{\widebar{E\ox f\ox F})}E\ox D\ox F) \sim (E\ox C \ox F\xra{E\ox \bar{f}\ox F} E\ox D \ox F) $$
\item[(ii)] for $f\in \CC (C,D), g\in \CC (D,E)$, 
$$(C\xra{\widebar{1_C}}C) \sim (C\xra{1_{C}}C) \ \text{  and  } \  
(C\xra{\widebar{gf}}E) \sim (C\xra{\bar{f}} D \xra{\bar{g}}E)$$ 
\item[(iii)] for all edges $U\xra{a}V, W\xra{b}X$ of $\CG$,
\begin{eqnarray*}
{(U\ox W\xra{a \ox W}V\ox W\xra{V\ox b}V\ox X) 
\sim (U\ox W\xra{U\ox b}U\ox X\xra{a\ox X}V\ox X)}
\end{eqnarray*} 
\item[(iv)] $(J\xra{J\ox \eta \ox \mathbb{I}} J\ox J^{\vee} \ox J\xra{(\mathbb{I} \ox \varepsilon \ox J)} J) \sim 1_J$
\item[(v)] $(J^{\vee}\xra{(\mathbb{I}\ox \eta \ox J^{\vee})} J^{\vee}\ox J\ox J^{\vee}\xra{J^{\vee}\ox \varepsilon\ox \mathbb{I}} J^{\vee}) \sim 1_{J^{\vee}}$
\item[(vi)] if $\ell \sim \ell '$ is any of the relations in (i)-(v), then 
$$Y\ox \ell \ox Z \sim Y\ox \ell '\ox Z \ .$$  
\end{itemize} 
Put $\CC[J^{\vee}] = \CF \CG/\sim$ as a category. 
We write $\llbracket\ell\rrbracket : U \to V$ for the equivalence class of a morphism $\ell : U\to V$ in $\CF \CG$.

Put $T\ox \llbracket \ell \rrbracket \ox W = \llbracket T\ox \ell \ox W\rrbracket$.
We can write $\llbracket a\rrbracket\ox \llbracket b\rrbracket : U\ox W \to V\ox X$ for the equivalence class of both sides of relation (iii),
and this extends by composition to morphisms of $\CF \CG$ in place of $a, b$.  
It is because of (i), (ii), (iii) and (vi) that $\CC[J^{\vee}]$ is strict monoidal under $\ox$ with unit $\mathbb{I}$ and
the inclusion functor $\Omega : \CC\to \CC[J^{\vee}]$, defined by $\Omega C = C, \Omega f=\llbracket \bar{f}\rrbracket$,
is strict monoidal. 
By virtue of (iv) and (v), we have $J\dashv J^{\vee}$ in $\CC[J^{\vee}]$, and we will write the counit and unit as $\varepsilon$ and $\eta$ rather than $\llbracket \varepsilon \rrbracket$ and $\llbracket \eta\rrbracket$. 

\section{Universality of $\CC[J^{\vee}]$}\label{UniversalityC[J*]}

Let us define a 2-category $\mathfrak{sJ}$. The objects $(\CC,J)$ are strict monoidal
categories $\CC$ equipped with a distinguished object $J\in \CC$. 
The 1-morphisms $(F,\phi) : (\CC,J) \to (\CC',J')$ consist of a strict monoidal functor $F : \CC\to \CC'$ equipped with an isomorphism $\phi : FJ\cong J'$ in $\CC'$.
The 2-morphisms $\sigma : (F,\phi)\Ra (G,\psi) : (\CC,J) \to (\CC',J')$ are monoidal natural transformations
$\sigma : F\Ra G : \CC \to \CC'$ such that $\psi \circ \sigma_J = \phi$.

Let us also define a 2-category $\mathfrak{sJ}_{\vee}$. The objects $(\CC,J,K,\varepsilon)$ are strict monoidal
categories $\CC$ equipped with a distinguished duality $J\dashv K$ with its counit $\varepsilon$. 
The 1-morphisms $(F, \phi) : (\CC,J,K,\varepsilon) \to (\CC',J',K',\varepsilon')$ consist of a strict monoidal functor $F : \CC\to \CC'$ equipped with an isomorphism $\phi : FJ\cong J'$ in $\CC'$;
notice then that the mate of $\phi^{-1}$ is an isomorphism $FK\cong K'$.
The 2-morphisms $\sigma : (F, \phi)\Ra (G,\psi) : (\CC,J,K,\varepsilon) \to (\CC',J',K',\varepsilon')$ are monoidal natural transformations $\sigma : F\Ra G : \CC \to \CC'$ such that $\psi \circ \sigma_J = \phi$. 

There is a forgetful 2-functor $\mathfrak{sU} : \mathfrak{sJ}_{\vee} \to \mathfrak{sJ}$ taking
$\sigma : (F, \phi)\Ra (G,\psi) : (\CC,J,K,\varepsilon) \to (\CC',J',K',\varepsilon')$ to
$\sigma : (F, \phi)\Ra (G,\psi) : (\CC,J) \to (\CC',J')$.

\begin{proposition}\label{strictunivprop}
The 2-functor $\mathfrak{sU} : \mathfrak{sJ}_{\vee} \to \mathfrak{sJ}$ has a left biadjoint
whose value at $(\CC, J)$ is $(\CC[J^{\vee}], J, J^{\vee}, \varepsilon)$.
Moreover, $\Omega : (\CC, J) \to (\CC[J^{\vee}],J)$ is the component of the unit of the biadjunction. 
\end{proposition}
\begin{proof}
Using \eqref{oxdecomp} and the nature of the relations (i)-(vi), we will prove that restriction
along $\Omega$ defines an equivalence of categories
\begin{eqnarray}\label{biadjnproperty}
\mathfrak{sJ}_{\vee}((\CC[J^{\vee}], J, J^{\vee}, \varepsilon),(\CX,H,K,\alpha)) \simeq 
\mathfrak{sJ}((\CC,J),(\CX,H))
\end{eqnarray}
which is surjective on objects.
To prove this surjectivity, take a strict monoidal functor $F : \CC\to \CX$ and $\phi : FJ \cong H$. 
Let $\beta : \mathbb{I} \to K\ox H$ be the unit corresponding to the counit $\alpha$ for the duality $H\dashv K$.
Then $\alpha'= \alpha \circ (\phi\ox K)$ and $\beta' = (K\ox \phi^{-1})\circ \beta$ are counit and unit for
a duality $FJ\dashv K$.
We can define a graph morphism $F' : \CG\to \CX$ by 
\begin{eqnarray*}
& U\mapsto FU_0\ox K\ox FU_1\ox K\ox \dots \ox K\ox FU_m   \\
\varepsilon \mapsto \alpha', & \eta \mapsto \beta', \ f \mapsto Ff, \ (U,p,V) \mapsto F'U\ox F'p\ox F'V 
\end{eqnarray*}
which we can immediately extend to a functor $\hat{F} : \CF \CG\to \CX$
by the universal property of the domain free category. 
Relations (i)-(iii) and (vi) are preserved by $\hat{F}$ since $F$ is strict monoidal;
relations (iv)-(v) are preserved since $\alpha'$ and $\beta'$ are counit and unit
for $FJ\dashv K$. So a functor $\bar{F} : \CC[J^{\vee}] \to \CX$ is induced.
Notice that $F'$ is a monoid morphism on the monoids of vertices of $\CG$ and the underlying graph of $\CX$ under tensor. With this, relations (iii) then imply that $\bar{F}$ is strict monoidal. By construction, we have $F = \bar{F}\Omega$. We also have $\phi : \bar{F}J = FJ \cong H$.

It remains to prove that restriction along $\Omega$ gives the fullness and faithfulness required for \eqref{biadjnproperty}. Take $(F, \phi), (G,\psi) : (\CC[J^\vee],J,J^\vee,\varepsilon) \to (\CX,H,K,\varepsilon)$
in $\mathfrak{sJ}_{\vee}$ and a monoidal natural transformation $\sigma : F\Omega \Ra G\Omega : \CC \to \CX$ with $\psi \circ \sigma_J = \phi$. 
We need to see that there is a unique $\bar{\sigma} : (F, \phi)\Ra (G,\psi)$ with $\bar{\sigma}\Omega=\sigma$.

 Since $\bar{\sigma}$ is to be monoidal and natural, we have commutativity of 
\begin{equation*}
\xymatrix{
FJ\ox FJ^{\vee} \ar[rr]^-{F\varepsilon} \ar[d]_-{\sigma_J\ox \bar{\sigma}_{J^{\vee}}} && \mathbb{I} \ar[d]^-{1_{\mathbb{I}}} \\
GJ\ox GJ^{\vee} \ar[rr]_-{G\varepsilon'} && \mathbb{I}}
\end{equation*} 
which shows that $\bar{\sigma}_{J^{\vee}}$ is forced to be the mate of 
$\sigma_J^{-1}= \phi^{-1}\circ \psi$. With this, we are now forced to put
\begin{eqnarray*}
\bar{\sigma}_U = \sigma_{U_0}\ox \bar{\sigma}_{J^{\vee}}\ox \sigma_{U_1}\ox \bar{\sigma}_{J^{\vee}}\ox \dots \ox \bar{\sigma}_{J^{\vee}}\ox \sigma_{U_m} \ .   
\end{eqnarray*}
With this as definition, the monoidality condition is obvious and naturality at equivalence classes of primitive morphisms is straightforward. We have the desired unique extension $\bar{\sigma}$. 
\end{proof}

Let us define a 2-category $\mathfrak{J}$. The objects $(\CC,J)$ are monoidal
categories $\CC$ equipped with a distinguished object $J\in \CC$. 
The 1-morphisms $(F,\phi) : (\CC,J) \to (\CC',J')$ consist of a strong monoidal functor $F : \CC\to \CC'$ equipped with an isomorphism $\phi : FJ\cong J'$ in $\CC'$.
The 2-morphisms $\sigma : (F,\phi)\Ra (G,\psi) : (\CC,J) \to (\CC',J')$ are monoidal natural transformations
$\sigma : F\Ra G : \CC \to \CC'$ such that $\psi \circ \sigma_J = \phi$.

Predictably, we also define a 2-category $\mathfrak{J}_{\vee}$. It is the full sub-2-category of $\mathfrak{J}$
consisting of those objects $(\CC,J)$ for which the distinguished object $J$ has a right dual.  
Let $\mathfrak{U} : \mathfrak{J}_{\vee} \to \mathfrak{J}$ be the inclusion 2-functor.

\begin{corollary}\label{flex}
The 2-functor $\mathfrak{U} : \mathfrak{J}_{\vee} \to \mathfrak{J}$ has a left biadjoint
whose value at $(\CC, J)$ is $(\CC[J^{\vee}], J)$.
Moreover, $\Omega : (\CC, J) \to (\CC[J^{\vee}],J)$ is the component of the unit of the biadjunction. 
\end{corollary}
\begin{proof}
This is a standard consequence of Proposition~\ref{strictunivprop} when dealing with flexible categorical structures.
\end{proof}

\begin{remark}\label{treermk}
Consider the case where $\CC$ is the free strict monoidal category on a
single generating object; that is, it is the discrete category $\mathbb{N}$ with addition
as tensor product. 
Take $J$ to be the natural number $1$. 
We write $\mathbb{N}[1^{\vee}]$ for $\CC[J^{\vee}]$ in this case. 
Objects of $\mathbb{N}[1^{\vee}]$ can be identified with plane rooted trees of height $2$ (in the sense of \cite{Batanin1998}) and having at least one edge attached to the root. Tensor product in terms of rooted trees is as shown
in diagram \eqref{egtreeten}.
\begin{eqnarray}\label{egtreeten}
\begin{aligned}
\psscalebox{0.6 0.6} % Change this value to rescale the drawing.
{
\begin{pspicture}(0,-1.40398)(17.051266,1.40398)
\psline[linecolor=black, linewidth=0.04](0.017122956,1.3898379)(1.657123,-1.330162)
\psline[linecolor=black, linewidth=0.04](0.91712296,-0.09016208)(1.757123,1.3698379)
\psline[linecolor=black, linewidth=0.04](1.657123,-1.330162)(2.437123,0.009687537)(2.4165967,1.3698379)
\rput[bl](3.6,-0.07016209){\LARGE{$\otimes$}}
\psline[linecolor=black, linewidth=0.04](5.177123,1.3498379)(6.757123,-1.3701621)(9.197123,1.3498379)
\psline[linecolor=black, linewidth=0.04](5.917123,0.049837913)(6.997123,1.3298379)
\psline[linecolor=black, linewidth=0.04](5.937123,0.029837914)(5.997123,1.3498379)
\psline[linecolor=black, linewidth=0.04](8.017123,0.029837914)(7.577123,1.3298379)
\rput[bl](9.5,0.009837913){\LARGE{$=$}}
\psline[linecolor=black, linewidth=0.04](10.537123,1.3698379)(13.437123,-1.2501621)(17.017122,1.3698379)(17.037123,1.3898379)
\psline[linecolor=black, linewidth=0.04](15.457123,0.18983792)(15.197123,1.3498379)
\psline[linecolor=black, linewidth=0.04](11.817123,0.24983792)(11.697123,1.3698379)
\psline[linecolor=black, linewidth=0.04](13.417123,0.28983793)(13.457123,-1.2501621)
\psline[linecolor=black, linewidth=0.04](12.197123,1.3298379)(13.417123,0.24983792)
\psline[linecolor=black, linewidth=0.04](12.917123,1.3298379)(13.417123,0.26983792)(13.797123,1.3498379)
\psline[linecolor=black, linewidth=0.04](13.457123,0.28983793)(14.517123,1.3698379)
\end{pspicture}
}
\end{aligned}
\end{eqnarray}
Notice that $\mathbb{N}[1^{\vee}]$ and the skeletal version of the 
monoidal category $\mathrm{Dpr}$
have the same universal property: strict monoidal functors from it
into a strict monoidal category $\CX$ are in bijection with duality pairs
in $\CX$. So the monoids of objects should be isomorphic.
The monoid of objects of this $\mathrm{Dpr}$ is the free
monoid $\{-,+\}^*$ on two symbols. It is easy to see directly that the
monoid morphism $\gamma : \{-,+\}^*\to \mathbb{N}[1^{\vee}]$ defined by
\begin{eqnarray*}
\begin{aligned}
\psscalebox{0.6 0.6} % Change this value to rescale the drawing.
{
\begin{pspicture}(0,-1.0533887)(10.875672,1.0533887)
\psline[linecolor=black, linewidth=0.04](9.4,0.00972786)(10.1,-0.95027214)(10.86,0.00972786)
\psline[linecolor=black, linewidth=0.04](3.14,-1.0502721)(3.3,-0.030272141)(3.1,1.0497279)
\rput[bl](0.0,-0.11027214){\Large{$\gamma(-) \ =$}}
\rput[bl](5.76,-0.11027214){\Large{$\gamma(+) \ =$}}
\end{pspicture}
}
\end{aligned}
\end{eqnarray*}
is bijective. 
 
\end{remark}

\section{Fullness of $\Omega : \CC\to \CC[J^{\vee}]$}\label{Omegafullsection}

\begin{proposition}\label{Omegafull}
The inclusion functor $\Omega : \CC \to \CC[J^{\vee}]$ is full.  
\begin{proof}
Suppose $\ell : C \to D$ is a morphism in $\CC[J^{\vee}]$ with $C, D\in \CC$.
If $\ell$ has a representative path that passes through no vertex with a $J^{\vee}$ factor then
relations (i), (ii), (iii) imply that $\ell = \llbracket f\rrbracket$ for some $f\in \CC(C,D)$.
The only way that factors $J^{\vee}$ can be created in vertices in the path is by using the 
primitive edge $\eta$
and the only way they can be removed is by using the primitive edge $\varepsilon$.
So suppose our path representing $\ell$ has a vertex with a $J^{\vee}$ factor.
Then there must be an edge in the path of the form $ U\ox V \xra{U\ox \eta \ox V} U\ox J^{\vee}\ox J\ox V$. In order that the created $J^{\vee}$ in the target can be removed, the path must continue on as
\begin{eqnarray*}
U\ox J^{\vee}\ox J\ox V\xra{a\ox J^{\vee}\ox b} X\ox J\ox J^{\vee}\ox W\xra{X\ox \varepsilon \ox W}X\ox W \ .
\end{eqnarray*}
However, using relations (ii) and (iii), we have
\begin{eqnarray*}\lefteqn{
(X\ox \varepsilon\ox W)(a\ox J^{\vee}\ox b)(U\ox \eta \ox V)} \\
& \sim & (X\ox \varepsilon\ox W)(X\ox J\ox J^{\vee}\ox b)(a\ox J\ox V)(U\ox \eta \ox V) \\
& \sim & (X\ox b)(X\ox \varepsilon \ox J\ox V)(X\ox J\ox \eta\ox V)(a\ox V) \ . 
\end{eqnarray*}
Invoking relation (iv), we are left with $U\ox V\xra{a\ox V}X\ox J\ox V\xra{X\ox b}X\ox W$.
By induction, any path containing edges involving an $\eta$ is related to one not involving any.
We are back to the first sentence of the proof. This proves $\Omega$ is full.
\end{proof}
\end{proposition}

\begin{remark} 
Relation (v) was not used in this proof.
\end{remark}
 
\section{Faithfulness of $\Omega : \CC\to \CC [J^{\vee}]$}\label{Omegafaithfulsection}

Recall the bicategory $\mathrm{Mod}$ of categories and modules (called ``bimodules'' by Lawvere \cite{LawMetric} and ``distributors'' by B\'enabou \cite{Dist}) between them.
The objects are small categories. The hom categories are presheaf categories:
\begin{eqnarray*}
\mathrm{Mod}(\CA,\CB) = [\CB^{\mathrm{op}}\times \CA,\mathrm{Set}] \ .
\end{eqnarray*}
Composition is tensor product of modules:
\begin{eqnarray*}
(N\ox_{\CB}M)(C,A) =  \int^B{M(B,A)\times N(C,B)} \ .
\end{eqnarray*}
The identity module of $\CA$ will be denoted $H_{\CA}$ and is the hom presheaf: 
$H_{\CA}(A',A) = \CA(A',A)$. 
All right liftings and right extensions exist in the bicategory $\mathrm{Mod}$. 
For every functor $F : \CA\to \CB$, we obtain a module $F_* : \CA\to \CB$
and a module $F^* : \CB\to \CA$ with an adjunction $F_*\dashv F^*$ in the bicategory
$\mathrm{Mod}$; we have 
\begin{eqnarray*}
F_*(B,A) = \CB(B,FA) \ \text{ and } \ F^*(A,B) = \CB(FA,B) \ .
\end{eqnarray*}

So then we have, for each category $\CC$, a closed monoidal category $\mathrm{Mod}(\CC,\CC)$;
the tensor product is $\ox_{\CC}$ with unit $H_{\CC}$. If $\CC$ is monoidal then each object $C\in \CC$ determines a functor $X\ox- : \CC\to \CC$ and so a duality
\begin{eqnarray*}
(X\ox -)_*\dashv (X\ox -)^*
\end{eqnarray*}
in the monoidal category $\mathrm{Mod}(\CC,\CC)$.
This leads to the {\em Cayley functor} 
\begin{eqnarray}\label{Upsilon}
\Upsilon : \CC \to \mathrm{Mod}(\CC,\CC)
\end{eqnarray}
defined by
$$\Upsilon X(Z,Y) = (X\ox -)_*(Z,Y) = \CC (Z,X\ox Y) \ ,$$
while the effect on homs is the function 
\begin{eqnarray*}
\Upsilon : \CC (X,X') \to \int_{Z,Y}[\CC (Z,X\ox Y),\CC (Z,X'\ox Y)]
\end{eqnarray*}
taking $X\xra{f}X'$ to the family of functions $\CC (Z,X\ox Y)\xra{\CC (Z,f\ox Y)}\CC (Z,X'\ox Y)$. 
So $\Upsilon$ is {\em faithful and conservative}.  

We also have
$$\Upsilon^{\vee} X(Z,Y) = (X\ox -)^*(Z,Y) = \CC (X\ox Z,Y) $$
and the duality
$$\Upsilon X \dashv \Upsilon^{\vee}X$$
in $\mathrm{Mod}(\CC,\CC)$. 

What is more, $\Upsilon$ is strong monoidal. We have
$$\Upsilon \mathbb{I}(Z,Y) = \CC(Z,\mathbb{I}\ox Y) = H_{\CC}(Z,Y) $$
and
\begin{eqnarray*}\lefteqn{
(\Upsilon X\ox_{\CC}\Upsilon Y)(C,E)
 =  \int^{D}{(\Upsilon Y)(D,E)\times (\Upsilon X)(C,D)}} \\
& = & \int^D{\CC(D,Y\ox E)\times \CC(C,X\ox D)}
\cong  \CC(C, X\ox Y \ox E) \\
& = & \Upsilon (X\ox Y) \ .
 \end{eqnarray*}

By Corollary~\ref{flex}, there exists a strong monoidal functor
\begin{eqnarray}\label{Gamma}
\Gamma : \CC[J^{\vee}] \to \mathrm{Mod}(\CC,\CC)
\end{eqnarray}
such that
$\Gamma \Omega \cong \Upsilon$ and $\Gamma$ preserves the
duality $J\dashv J^{\vee}$; so $\Gamma J^{\vee} = \Upsilon^{\vee}J$
and $\Gamma\eta: H_{\CC}\Rightarrow \Upsilon^\vee J\otimes_{\CC}\Upsilon J$ is such that triangle \eqref{Gamma_eta} commutes for all $X,Z\in\ob\CC$.
\begin{eqnarray}\label{Gamma_eta}
\begin{aligned}
\xymatrix{
\CC(X,Z) \ar[rd]_{J\ox -}\ar[rr]^{\Gamma\eta}   && \int^Y{\CC(J\ox X,Y)\times \CC(Y,J\ox Z) }\ar[ld]^{\cong} \\
& \CC(J\ox X,J\ox Z)  &
}
\end{aligned}
\end{eqnarray}
 
\begin{theorem}\label{Omegaff}
The inclusion functor $\Omega : \CC \to \CC[J^{\vee}]$ is fully faithful.
\end{theorem}
\begin{proof}
Since $\Gamma \Omega \cong \Upsilon$, the fact that $\Upsilon$ is faithful implies $\Omega$ is too. We already know that $\Omega$ is full by Proposition~\ref{Omegafull}.
\end{proof}

\section{Coends for some homs of $\CC[J^{\vee}]$}\label{coends}

Some of the hom sets of $\CC[J^{\vee}]$ can be expressed as iterated coends over $\CC$, 
namely those which hom into or out of an object of $\CC$. 
\bigskip

Take objects $U : [m]\to \ob \CC$ and $V : [n]\to \ob \CC$.

For $m=n=0$, by Theorem~\ref{Omegaff}, we have 
\begin{eqnarray}\label{m=n=0}
\CC[J^{\vee}](U,V) \cong \CC(U_0,V_0) \ .
\end{eqnarray}

For $m>0, n=0$, there is a function
\begin{eqnarray}\label{m>0,n=0}
 \int^{X\in \CC}{\CC[J^{\vee}](U\partial_m,X\ox J)\times \CC(X\ox U_m,V_0)} \xra{\zeta_{m,0}} \CC[J^{\vee}](U,V)
\end{eqnarray}
taking the equivalence class of $(U\partial_m\xra{f}X\ox J, X\ox U_m\xra{g}V_0)$ to the composite
\begin{eqnarray*}
U=U\partial_m\ox J^{\vee}\ox U_m \xra{f\ox 1\ox 1} X\ox J\ox J^{\vee}\ox U_m\xra{1\ox \varepsilon \ox 1}X\ox U_m\xra{g} V_0=V \ .
\end{eqnarray*}

For $m=0, n>0$, there is a function
\begin{eqnarray}\label{m=0,n>0}
\int^{Y\in \CC}{\CC(U_0,V_0\ox Y)\times \CC[J^{\vee}](J\ox Y,V\partial_0)}\xra{\zeta_{0,n}}\CC[J^{\vee}](U,V)
\end{eqnarray}
taking the equivalence class of $(U_0\xra{f}V_0\ox Y, J\ox Y\xra{g}V\partial_0)$ to the composite
\begin{eqnarray*}
U=U_0 \xra{f} V_0\ox Y\xra{1\ox \eta \ox 1}V_0\ox J^{\vee}\ox J\ox Y\xra{1\ox 1\ox g} V_0\ox J^{\vee}\ox V\partial_0=V \ .
\end{eqnarray*}

For $m>0, n>0$, there is a function
\begin{eqnarray}\label{m,n>0}
\begin{aligned}
\int^{X,Y\in \CC}{\CC[J^{\vee}](U\partial_m,X\ox J)\times \CC(X\ox U_m,V_0\ox Y)\times \CC[J^{\vee}](J\ox Y,V\partial_0)} \\
\xra{\zeta_{m,n}}  \CC[J^{\vee}](U,V) 
\end{aligned} 
\end{eqnarray}
taking the equivalence class of $(U\partial_m\xra{f}X\ox J, X\ox U_m\xra{g}V_0\ox Y, J\ox Y\xra{h}V\partial_0)$ to the composite
\begin{eqnarray*}
U=U\partial_m\ox J^{\vee}\ox U_m \xra{f\ox 1\ox 1} X\ox J\ox J^{\vee}\ox U_m\xra{1\ox \varepsilon \ox 1}X\ox U_m\xra{g} V_0\ox Y \\
\xra{1\ox \eta\ox 1}V_0\ox J^{\vee}\ox J\ox Y\xra{1\ox 1\ox h} V_0\ox J^{\vee}\ox V\partial_0=V \ .
\end{eqnarray*}

Notice that $\zeta_{m,n}$ \eqref{m,n>0} is definitely not surjective in general.
For example, in the case of $\mathbb{N}[1^{\vee}]$ as in Remark~\ref{treermk},
the coend is empty for $U=V=1^{\vee}$. More generally, for most
morphisms of the form $f : J\ox X \to Y\ox J$ in $\CC$, the composite
\begin{eqnarray*}
X\ox J^{\vee}\xra{\eta \ox 1\ox 1} J^{\vee}\ox J\ox X\ox J^{\vee}\xra{1\ox f\ox 1}J^{\vee}\ox Y\ox J\ox J^{\vee}\xra{1\ox 1\ox \varepsilon} J^{\vee}\ox Y
\end{eqnarray*}
will not be in the image of $\zeta_{m,n}$.

We claim that $\zeta_{m,n}$ is invertible when either $m$ or $n$ is $0$. 

Let us look at the case of $\zeta_{0,n}$.
We begin by relating its domain to the functor $\Gamma$ of \eqref{Gamma}.
If $A\in \CC$ and $W : [p] \to \CC$ in $\CC[J^{\vee}]$ then
\begin{eqnarray}\label{bigcoend}
\lefteqn{
(\Gamma W)(A,\mathbb{I})}  \nonumber \\
& = & (\Gamma W_0\ox_{\CC}\Gamma J^{\vee}\ox_{\CC}\Gamma W_1\ox_{\CC}\Gamma J^{\vee}\ox_{\CC}\dots \ox_{\CC}\Gamma W_p) (A,\mathbb{I}) \nonumber \\
& = & (\Upsilon W_0\ox_{\CC}\Upsilon^{\vee} J\ox_{\CC}\Upsilon W_1\ox_{\CC}\Upsilon^{\vee} J\ox_{\CC}\dots \ox_{\CC}\Upsilon W_p) (A,\mathbb{I}) \\
 &\cong &  \int^{Y_1\dots Y_p\in \CC}{\CC (A,W_0\ox Y_1)\times \CC (J\ox Y_1,W_1\ox Y_2)\times\dots\times \CC (J\ox Y_p,W_p) \nonumber}
\end{eqnarray} 
where the last step involves the coend form of the Yoneda Lemma
and $W_p\ox \mathbb{I} = W_p$. Define the function
\begin{eqnarray}\label{zeta0.components}
\zeta_{0,\centerdot} : \Gamma W(A,\mathbb{I}) \lra \CC[J^{\vee}](A,W)
\end{eqnarray}
to be the composite of the isomorphism \eqref{bigcoend} with the function which
takes the equivalence class of 
\begin{eqnarray*}
(A\xra{f_0}W_0\ox Y_1, J\ox Y_1\xra{f_1}W_1\ox Y_2,\dots, J\ox Y_p\xra{f_p}W_p)
\end{eqnarray*}
to the composite $\zeta_{0,\centerdot}[f_0,f_1,\dots ,f_p] : =$ 
\begin{eqnarray*}
A\xra{f_0}W_0\ox Y_1\xra{1\ox\eta\ox 1}W_0\ox J^{\vee}\ox J\ox Y_1 \xra{1\ox 1\ox f_2}W_0\ox J^{\vee}\ox W_1\ox Y_2\ra \dots \\
\xra{1\ox\eta\ox 1} W\partial_p\ox J^{\vee}\ox J\ox Y_p\xra{1\ox f_p} W \ .
\end{eqnarray*}
It is clear that the above is well-defined and moreoever, the functions \eqref{zeta0.components} are the components of
a natural transformation 
\begin{eqnarray}\label{zeta0.nattran}
\begin{aligned}
\xymatrix{
\CC[J^{\vee}]\ar[rd]_{\tilde{\Omega}}^(0.5){\phantom{aaaa}}="1" \ar[rr]^{\Gamma}  && \mathrm{Mod}(\CC,\CC) \ar[ld]^{\mathrm{Ev}_{\mathbb{I}}}_(0.5){\phantom{aaaa}}="2" \ar@{<=}"1";"2"_-{\zeta_{0,\centerdot}}
\\
& [\CC^{\mathrm{op}},\mathrm{Set}] 
}
\end{aligned}
\end{eqnarray}
where $\tilde{\Omega}W = \CC[J^{\vee}](\Omega -,W)$ and
$\mathrm{Ev}_{\mathbb{I}}M = M(-,\mathbb{I})$. 
 \begin{theorem}\label{m=0Thm}
The natural transformation $\zeta_{0,\centerdot}$ in \eqref{zeta0.nattran} is invertible. 
\end{theorem}
 \begin{proof}
We will show that the composite $\zeta'_{0,\centerdot} : =$
\begin{eqnarray*}
\CC[J^{\vee}](A,W)\xra{\Gamma}\int_{X,Z}{[\CC(X,A\ox Z),\Gamma W(X,Z)]} \cong \int_Z{\Gamma W(A\ox Z,Z)}\xra{\mathrm{Ev}_{\mathbb{I}}} \Gamma W(A,\mathbb{I})
\end{eqnarray*}
is the inverse to $\zeta_{0,\centerdot}$. The composite
\begin{eqnarray*}
\CC[J^{\vee}](A,W)\xra{\zeta'_{0,\centerdot}}\Gamma W(A,I)\xra{\zeta_{0,\centerdot}}\CC[J^{\vee}](A,W)
\end{eqnarray*}
is natural in $W$ and so, by Yoneda, is determined by the value
at the identity of $W = A$ under the composite
\begin{eqnarray*}
\CC[J^{\vee}](A,A)\xra{\zeta'_{0,\centerdot}}\CC(A,A)\xra{\zeta_{0,\centerdot}}\CC[J^{\vee}](A,A) \ .
\end{eqnarray*}
The composite takes $1_A$ to itself yielding $\zeta_{0,\centerdot}\circ \zeta'_{0,\centerdot} = 1_{\CC[J^{\vee}](A,W)}$.  

We now look at the composite
\begin{eqnarray*}
\Gamma W(A,\mathbb{I}) \xra{\zeta_{0,\centerdot}}\CC[J^{\vee}](A,W)\xra{\zeta'_{0,\centerdot}}\Gamma W(A,\mathbb{I})
\end{eqnarray*}
as transported across the isomorphism \eqref{bigcoend}.
(Abusing notation, we use the same symbols for the transported functions.)
This is performed by applying $\Gamma$ to the morphism $\zeta_{0,\centerdot}[f_0,f_1,\dots ,f_p]$ which involves morphisms
in the image of $\Omega$ and the morphism $\eta$ in $\CC[J^{\vee}]$.
Making use of $\Gamma\Omega =\Upsilon$, strong monoidality of $\Gamma$
and the value of $\Gamma$ on $\eta$ as given by \eqref{Gamma_eta},
we see with some effort that $\zeta'_{0,\centerdot}\zeta_{0,\centerdot}[f_0,f_1,\dots ,f_p]= [f_0,f_1,\dots ,f_p]$, as required. 

It may be helpful for us to give more detail in the case $p=1$ to indicate what is involved in general. We need to show that the composite
\begin{eqnarray*}
\int^{Y\in \CC}{\CC(A,B\ox Y)\times \CC(J\ox Y,C)}\xra{\zeta_{0,1}}\CC[J^{\vee}](A,B\ox J^{\vee}\ox C) \xra{\Gamma} \\
\int_{X,Z}{\Big[ \CC(X,A\ox Z),\int^{Y_1,Y_2}{\CC(X,B\ox Y_1)\times \CC(J\ox Y_1,Y_2)\times \CC(Y_2,C\ox Z)}\Big] } \xra{\cong}\\
\int_Z{\int^{Y_1}{\CC(A\ox Z,B\ox Y_1)\times \CC(J\ox Y_1,J\ox C\ox Z)}} \xra{\mathrm{Ev}_{\mathbb{I}}} \phantom{AAAA} \\
\int^{Y\in \CC}{\CC(A,B\ox Y)\times \CC(J\ox Y,C)} \phantom{AAAAAAAAAA}
\end{eqnarray*}
is the identity. (It is actually possible in this $p=1$ case to use Yoneda to check only that, taking $A=B\ox Y$ and $J\ox Y = C$, the equivalence class of $(1_A,1_C)$ goes to itself. However, this possibility
is not fully available for $p>1$.) Take $[A \xra{f} B\ox Y, J\ox Y\xra{g}C]$
in the domain. Under $\zeta_{0,1}$ it goes to the composite
$A\xra{f}B\ox Y\xra{1\ox \eta \ox 1}B\ox J^{\vee}\ox J \ox Y\xra{1\ox 1 \ox g} B\ox J^{\vee}\ox C$. So the component of the natural transformation $\Gamma \zeta_{0,1} [f,g]$ at $(X,Z)$ is the composite
\begin{eqnarray*}
\CC(X,A\ox Z)\xra{\CC(1,f\ox 1)}\CC(X,B\ox Y\ox Z)\xra{\Gamma(1\ox \eta \ox 1)}
\\ \int^{Y_1,Y_2}{\CC(X,B\ox Y_1)\times \CC(J\ox Y_1,J\ox Y_2)\times \CC(Y_2,Y\ox Z)}\xra{\cong} \\
\int^{Y_1}{\CC(X,B\ox Y_1)\times \CC(J\ox Y_1,J\ox Y\ox Z)}\xra{1\times \CC(1,g\ox 1)} \\
\int^{Y_1}{\CC(X,B\ox Y_1)\times \CC(J\ox Y_1,C\ox Z)} \phantom{AAAAAA}
\end{eqnarray*}
 which, using \eqref{Gamma_eta}, corresponds to the element 
 $$([A\ox Z\xra{f\ox 1}B\ox Y\ox Z,J\ox Y\ox Z\xra{g\ox 1}C\ox Z])_Z$$ of 
 $$\int_Z{\int^{Y_1}{\CC(A\ox Z,B\ox Y_1)\times \CC(J\ox Y_1,J\ox C\ox Z)}} \ .$$
 The component at the index $Z=\mathbb{I}$ is our original $[A \xra{f} B\ox Y, J\ox Y\xra{g}C]$.
 \end{proof}
\begin{corollary}
The morphism $\zeta_{0,n}$ \eqref{m=0,n>0} is invertible for all $n > 0$.
\end{corollary}
\begin{proof}
Using Theorem~\ref{m=0Thm} twice, we obtain isomorphisms
\begin{eqnarray*}
\lefteqn{
\int^{Y\in \CC}{\CC(U_0,V_0\ox Y)\times \CC[J^{\vee}](J\ox Y,V\partial_0)}} \\
& \cong & \int^{Y\in \CC}{\CC(U_0,V_0\ox Y)\times \Gamma(V\partial_0)(J\ox Y,\mathbb{I})} \\
& \cong & \int^{Y,Y'\in \CC}{(\Gamma V_0)(U_0,Y)\times (\Gamma J^{\vee})(Y,Y') \times \Gamma(V\partial_0)(Y',\mathbb{I})} \\
& \cong & \Gamma(V_0\ox J^{\vee} \ox V\partial_0)(U_0,\mathbb{I}) 
\cong \Gamma V(U_0,\mathbb{I}) \\
& \cong & \CC[J^{\vee}](U,V) 
\end{eqnarray*}
whose composite is $\zeta_{0,n}$.
\end{proof}

Similarly, we have functions $\Gamma U(\mathbb{I},A)\to \CC[J^{\vee}](U,\Omega A)$ which are the components of
a natural isomorphism 
\begin{eqnarray}\label{zetam.0attran}
\begin{aligned}
\xymatrix{
\CC[J^{\vee}]\ar[rd]_{\Omega^{\dagger}}^(0.5){\phantom{aaaa}}="1" \ar[rr]^{\Gamma}  && \mathrm{Mod}(\CC,\CC) \ar[ld]^{\mathrm{Ev}'_{\mathbb{I}}}_(0.5){\phantom{aaaa}}="2" \ar@{<=}"1";"2"_-{\zeta_{\centerdot,0}}
\\
& [\CC,\mathrm{Set}]^{\mathrm{op}} 
}
\end{aligned}
\end{eqnarray}
where $\Omega^{\dagger}W = \CC[J^{\vee}](W, \Omega -)$ and
$\mathrm{Ev}'_{\mathbb{I}}M = M(\mathbb{I},-)$.

\section{$\CC[J^{\vee}]$ from $\mathbb{N}[1^{\vee}] \simeq \mathrm{Dpr}$}\label{C[J^]from[N1^]}\label{RemPushN}
We use the notation of Remark~\ref{treermk}.
The (bicategorical) initial object of $\mathfrak{J}$ is $(\mathbb{N}, 1)$
while that of $\mathfrak{J}_{\vee}$ is $(\mathbb{N}[1^{\vee}], 1)$.
For any $(\CC,J)\in \mathfrak{J}_{\vee}$, we write $J^{\vee}$ for a choice of right dual for $J$.

Binary (bicategorical) coproduct in $\mathfrak{J}$ will be denoted by 
\begin{eqnarray*}
(\CC,J)\xra{\mathrm{in}_1} (\CC +_{\mathbb{N}} \CD,J) \xla{\mathrm{in}_2} (\CD,K) \ .
\end{eqnarray*}
We point out that, if either $J\in \CC$ or $K\in \CD$ has a right dual then $J\in \CC +_{\mathbb{N}} \CD$
has a right dual. 
\begin{proposition}\label{coprodformula}
For each object $(\CC,J)\in \mathfrak{J}$, there is an equivalence 
\begin{eqnarray*}
(\CC[J^{\vee}], J)\simeq (\mathbb{N}[1^{\vee}] +_{\mathbb{N}} \CC,1)
\end{eqnarray*}
in $\mathfrak{J}_{\vee}$.
\end{proposition}
\begin{proof}
Take $(\CX,K)\in \mathfrak{J}_{\vee}$. We have pseudonatural equivalences
\begin{eqnarray*}\lefteqn{
\mathfrak{J}_{\vee}((\mathbb{N}[1^{\vee}] +_{\mathbb{N}} \CC,1), (\CX,K))} \\
& \simeq & \mathfrak{J}((\mathbb{N}[1^{\vee}] +_{\mathbb{N}} \CC,1), (\CX,K)) \\
& \simeq & \mathfrak{J}((\mathbb{N}[1^{\vee}],1) , (\CX,K))\times  \mathfrak{J}((\CC,J), (\CX,K)) \\
& \simeq & \mathbf{1} \times  \mathfrak{J}((\CC,J), (\CX,K)) \simeq \mathfrak{J}((\CC,J), (\CX,K)) \\
& \simeq &  \mathfrak{J}_{\vee}((\CC[J^{\vee}], J), (\CX,K)) \ .
 \end{eqnarray*}
 The result now follows using the bicategorical Yoneda lemma \cite{14}.
\end{proof}

\section{A comment on enriched versions}\label{Klinear}

For brevity we did not pursue the enriched version of our results. In this section we briefly comment on the $K$-linear setting, for a commutative ring $K$. We say that a category is  monoidal $K$-linear if it comes equipped with a $K$-linear and a monoidal structure where $-\otimes-$ is $K$-linear in each variable.

Starting from a $K$-linear monoidal category $\CC$ and $J\in\CC$, Section~\ref{ConstructionC[J*]} still gives us a monoid $(\CG_0,\otimes,\mathbb{I})$ and graph $\CG=(\CG_0,\CG_1)$. Then we can form the free $K$-linear category $\CF_K\CG$ on the graph. Morphisms in $\CF_K\CG$ are formal $K$-linear combinations of morphisms in $\CF\CG$. We then consider the same relations (i)-(vi) from Section~\ref{ConstructionC[J*]} on $\CF_K\CG$ as we did on $\CF\CG$ with only modification that (i) needs to be extended to
$$\overline{E\otimes (\lambda f+\mu g)\otimes F}\;\sim\; \lambda E\otimes \bar{f}\otimes F +\mu E\otimes\bar{g}\otimes F,$$
for $\lambda,\mu\in K$ and $f,g\in\CC(C,D)$. We then denote again by $\CC[J^\vee]$ the corresponding quotient category of $\CF_K\CG$. That $\CC[J^\vee]$ is $K$-linear monoidal such that the canonical strict monoidal functor $\CC\to\CC[J^\vee]$ (which now is $K$-linear) satisfies the corresponding universal property follows exactly as before. Also the proof of the fullness of $\CC\to \CC[J^\vee]$ carries over verbatim. The constructions in Section~\ref{Omegafaithfulsection} are written such that the enriched version is obtained by trivial replacement of the category Set. In particular $\mathrm{Mod}(\CC,\CC)$ can be defined as the category of functors from $\CC^{\mathrm{op}}\times \CC$ to the category of $K$-modules which are $K$-linear in each variable, or as $K$-linear functors from $\CC^{\mathrm{op}}\otimes_K\CC$. Consequently $\CC\to \CC[J^\vee]$ is still fully faithful. Furthermore, the technique in Section~\ref{coends} shows for instance that we get an isomorphism of $K$-modules
$$\int^{X\in \CC}{\CC[J^{\vee}](U\partial_m,X\ox J)\otimes_K \CC(X\ox U_m,V)} \xra{\cong} \CC[J^{\vee}](U,V),$$
for $V\in\ob\CC$ and $U:[m]\to\ob\CC$.

Note that the definition of $\CC[J^\vee]$ in general actually depends on whether we view $\CC$ as a $K$-linear category or not. For instance, if $J$ has an endomorphism which is not a scalar multiple of the identity, the set of morphisms $\mathbb{I}\to J^\vee\otimes J$ depends on which definition we use.

\section{A comment on autonomisation}\label{Auton}
We briefly justify our focus on the adjoining of {\em one} (right) dual. Firstly, we can define $\CC[{}^\vee J]$ similarly to $\CC[J^\vee]$ by adjoining a left dual. In particular we can set $\CC[{}^\vee J]=(\CC^{\mathrm{rev}}[J^\vee])^{\mathrm{rev}}$, where $-^{\mathrm{rev}}$ denotes the reverse of a monoidal category, which switches the order of the tensor product. Using the explicit model of our categories $\CC[J^\vee]$ and via the principle of transfinite induction we can then construct an autonomous category from any strict monoidal category as a direct limit by iteratively adjoining left and right duals. 

It is instructive to break up the above procedure in two steps. In a first step we can adjoin all iterated left and right duals of a given object in a strict monoidal category. Just as in Section~\ref{RemPushN} one observes that this is actually a pseudopushout with respect to the free autonomous monoidal category on a single object of Example 3 in Section~\ref{Ioid}. Now we can construct an autonomous category from a small strict monoidal category $\CC$ by iteratively adjoining (at once) all duals of a given object
 along a well-order on $\ob\CC$.

We denote the category obtained in the above procedure by $\mathrm{Auton}\CC$. It follows from Theorem~\ref{Omegaff} and standard techniques that the canonical monoidal functor $\CC\to\mathrm{Auton}\CC$ is fully faithful. Furthermore, by construction and Corollary~\ref{flex}, the functor $\CC\to\mathrm{Auton}\CC$ is universal in the sense that it yields an equivalence between the categories of strong monoidal functors from $\CC$ and $\mathrm{Auton}\CC$ to any autonomous monoidal category. All of the above remains true in the $K$-linear setting.

\begin{center} 
--------------------------------------------------------
\end{center}

\appendix


\begin{thebibliography}{000}

\bibitem{Aud1974} Claude Auderset, \textit{Adjonctions et monades au niveau des 2-cat\'egories}, Cahiers de topologie et g\'eom\'etrie diff\'erentielle cat\'egoriques \textbf{15} (1974) 3--20.\label{Aud1974}

\bibitem{Batanin1998} Michael A. Batanin, \textit{Monoidal globular categories as a natural environment for the theory of weak $n$-categories}, Advances in Math. \textbf{136} (1998) 39--103.\label{Batanin1998}

\bibitem{Dist} Jean B\'enabou,  \textit{Les distributeurs}, Univ. Catholique de Louvain, S\'eminaires de Math. Pure, Rapport No.  \textbf{33} (1973).\label{Dist}

\bibitem{DayPas2008} Brian J. Day and Craig Pastro, \textit{Note on Frobenius monoidal functors}, New York J. Math. \textbf{14} (2008) 733--742.\label{DayPas2008}

\bibitem{Delp2019} Antonin Delpeuch, \textit{Autonomization of monoidal categories}, (15 Jun 2019) 25pp.; 
see arXiv:1411.3827v3. \label{Delp2019}

\bibitem{EilKel1966} Samuel Eilenberg and G. Max Kelly, \textit{Closed categories}, Proceedings of the Conference on Categorical Algebra (La Jolla, 1965), (Springer-Verlag,1966) 421--562.\label{EilKel1966}

\bibitem{xii} Andr\'e Joyal and Ross Street, \textit{Planar diagrams and tensor algebra} (handwritten notes, 1988);\\
see \url{http://web.science.mq.edu.au/~street/PlanarDiags.pdf}.\label{xii}

\bibitem{37} Andr\'e Joyal and Ross Street, \textit{The geometry of tensor calculus I}, Advances in Math. \textbf{88} (1991) 55--112.\label{37}

\bibitem{Kan1958} Daniel M. Kan, \textit{Adjoint functors}, Transactions of the American Mathematical Society \textbf{87} (1958) 294--329. \label{Kan1958}

\bibitem{Kel1972} G. Max Kelly, \textit{Many-variable functorial calculus I}, Lecture Notes
in Math. \textbf{281} (Springer-Verlag 1972) 66--105.\label{Kel1972}

\bibitem{118} Stephen Lack and Ross Street, \textit{Triangulations, orientals, and skew-monoidal categories}, Advances in Math. \textbf{258} (2014) 351--396.\label{118}

\bibitem{LawMetric} F. W. (Bill) Lawvere, \textit{Metric spaces, generalized logic and closed categories}, Reprints in Theory and Applications of Categories \textbf{1} (2002) pp.1--37; originally published as: \textit{Rendiconti del Seminario Matematico e Fisico di Milano} \textbf{53} (1973) 135--166.\label{LawMetric}

\bibitem{CWM} Saunders Mac Lane, \textit{Categories for the Working Mathematician}, Graduate Texts in Mathematics \textbf{5} (Springer-Verlag, 1971).\label{CWM} 

\bibitem{NSR1972} N. Saavedra Rivano, \textit{Cat\'egories Tannakiennes}, Lecture Notes
in Math. \textbf{265} (Springer-Verlag, 1972).\label{NSR1972}

\bibitem{28} Stephen Schanuel and Ross Street, \textit{The free adjunction}, Cahiers de topologie et g\'eom\'etrie diff\'erentielle cat\'egoriques \textbf{27} (1986) 81--83.\label{28}

\bibitem{Shum1994} Mei Chee Shum, \textit{Tortile tensor categories}, J. Pure Appl. Algebra \textbf{93} (1994) 57--110.\label{Shum1994}

\bibitem{14} Ross Street, \textit{Fibrations in bicategories}, Cahiers de topologie et g\'eom\'etrie diff\'erentielle \textbf{21} (1980) 111--160.\label{14}

\bibitem{68} Ross Street, \textit{Braids among the groups}, Seminarberichte aus dem Fachbereich Mathematik, \textbf{63(5)} (1998) 699--703.\label{68}

\bibitem{TempLieb} Neville Temperley and Elliott Lieb, \textit{Relations between the `percolation' and `colouring' problem and other graph-theoretical problems associated with regular planar lattices: some exact results for the `percolation' problem}, Proceedings of the Royal Society A: Mathematical, Physical and Engineering Sciences \textbf{322(1549)} (1971) 251--280.\label{TempLieb}

\bibitem{Yet1988} David N. Yetter, \textit{Markov Algebras}, in ``Braids'' (eds: J.S. Birman and A. Libgober), Contemporary Math. \textbf{78} (Amer. Math. Soc. 1988) 705--730.\label{Yet1988}

\end{thebibliography}
\end{document}